\documentclass[11pt,letterpaper]{article}

\usepackage[utf8x]{inputenx}
\usepackage[T1]{fontenc}
\usepackage{textcomp}
\usepackage{lmodern}

\usepackage{amsmath, amssymb}
\usepackage{amsthm}

\newtheorem{theorem}{Theorem}
\newtheorem{lemma}[theorem]{Lemma}
\newtheorem{proposition}[theorem]{Proposition}
\newtheorem{observation}[theorem]{Observation}
\newtheorem{question}[theorem]{Question}

\newtheorem{fact}[theorem]{Fact}

\usepackage{amsfonts}

\usepackage{graphics}
\usepackage{graphicx}
\usepackage{caption}
\usepackage{booktabs}
\usepackage{floatrow}
\usepackage{subfig}
\usepackage{psfrag}
\usepackage{epic,eepic}
\usepackage{color}
\usepackage{tikz}
\usetikzlibrary{knots}

\usepackage{pst-knot}

\usepackage[affil-sl]{authblk}

\usepackage{cite}
\usepackage{enumerate}
\usepackage{paralist}
\usepackage{amsrefs}   

\usepackage{fullpage}

\def\umin{{\uu_{\text{\rm min}}}}

\def\oB{\overline{B}}
\def\oR{\overline{R}}
\def\obb{\overline{b}}
\def\orr{\overline{r}}
\def\oD{\overline{D}}
\def\oS{\overline{S}}

\def\cy{{/\hbox{\hglue -0.12cm}/}}

\theoremstyle{plain}

\def\dd{{\mathcal D}}

\def\cc{{\mathcal C}}

\def\uu{{\mathcal U}}

\def\real{{\mathbb{R}}}

\def\ff{{\mathcal F}}

\def\ignore#1{{}}

\title{\bf On the number of unknot diagrams}

\author[1]{Carolina Medina}
\author[2,3]{Jorge L.~Ram\'{\i}rez-Alfons\'{\i}n}
\author[1,3]{Gelasio Salazar}

\affil[1]{Instituto de F\'\i sica, UASLP. San Luis Potos\'{\i}, Mexico, 78000.}
\affil[2]{Institut Montpelli\'erain Alexander Grothendieck, Universit\'e de Montpellier. Place Eug\`eene Bataillon, 34095 Montpellier, France.}
\affil[3]{Unit{\'e} Mixte Internationale CNRS-CONACYT-UNAM ``Laboratoire Solomon Lefschetz''. Cuernavaca, Mexico.}

\begin{document}

\maketitle

\begin{abstract}
Let $D$ be a knot diagram, and let $\dd$ denote the set of diagrams that can be obtained from $D$ by crossing exchanges. If $D$ has $n$ crossings, then $\dd$ consists of $2^n$ diagrams. A folklore argument shows that at least one of these $2^n$ diagrams is unknot, from which it follows that every diagram has finite unknotting number. It is easy to see that this argument can be used to show that actually $\dd$ has more than one unknot diagram, but it cannot yield more than  $4n$ unknot diagrams. We improve this linear bound to a superpolynomial bound, by showing that at least $2^{\sqrt[3]{n}}$ of the diagrams in $\dd$ are unknot. We also show that either all the diagrams in $\dd$ are unknot, or there is a diagram in $\dd$ that is a diagram of the trefoil knot. 
\end{abstract} 

\section{Introduction}\label{sec:intro}

We follow standard knot theory terminology as in~\cite{adams}. All knots under consideration are tame, that is, they have regular diagrams. All diagrams under consideration are regular. 

\subsection{Our main result}

A staple in every elementary knot theory course is that, given any diagram $D$, it is always possible to turn $D$ into a diagram of the unknot (that is, into an {\em unknot diagram}) by performing some crossing exchanges on $D$ (see Figure~\ref{fig:croexc}). The minimum number of such crossing exchanges that turn $D$ into an unknot diagram is a well-studied parameter, the {\em unknotting number} of $D$.

We review in Section~\ref{sec:background} the well-known argument that turns any diagram $D$ into an unknot diagram, by performing crossing exchanges. As we shall remark, this argument actually produces more than one unknot diagram from $D$, but it gives a relatively small number: if $D$ has $n$ crossings, this idea yields at most $4n$ distinct unknot diagrams that can be obtained from $D$ by crossing exchanges. Our main result is that actually there is always a superpolynomial number of such unknot diagrams.

\begin{theorem}\label{thm:maindia}
Let $D$ be a knot diagram with $n$ crossings. Then there are at least $2^{\sqrt[3]{n}}$ distinct unknot diagrams that can be obtained by performing crossing exchanges on $D$.
\end{theorem}

As far as we know, no non-trivial general lower bound along the lines of Theorem~\ref{thm:maindia} has been previously reported. There is, however, a substantial amount of work related to the complementary problem: out of the $2^n$ diagrams that can be obtained from $D$ by crossing exchanges, how many are knotted? This is related to the well-studied Frisch-Wasserman-Delbruck conjecture (see Section~\ref{sec:background}), which in its general form can be paraphrased as follows: if we take a random knot diagram with $n$ crossings, then the probability that it is knotted tends to $1$ as $n$ goes to infinity.

\begin{figure}[ht]
\centering
\scalebox{0.7}{\begin{tikzpicture}
\begin{scope}[line width=2pt, ]
\draw(-1.15,2.14).. controls (-0.8,3.23) and (0.8,3.3) .. (1.2,2);
\draw (0.12,0.075).. controls (1,0.5) and (1.4,1) .. (1.2,2);
\draw (0,-3).. controls (-2,-2.5) and (-1.48,-0.47) .. (-0.13,-0.06);
\draw(0,-3) .. controls (4,-4) and (3.9,1.56).. (1.35,1.98);
\draw (1.07,2.02) .. controls (0.7,2.1) and (-0.7,2.1).. (-1.2,2);
\draw(-0.15,-3.04) .. controls (-4,-3.8) and (-4,1.6).. (-1.2,2);
\draw (0.15,-2.95).. controls (1.9,-2.5) and (1.5,-0.5) .. (0,0);
\draw (0,0).. controls (-1,0.5) and (-1.4,1) .. (-1.22,1.85);
\end{scope}
\begin{scope} [white,shift={(0,0)}]
\filldraw (0.15,0.1) circle (3pt);
\filldraw (-0.15,-0.083) circle (3pt);
\filldraw (0.24,-2.89) circle (3pt);
\filldraw (-0.25,-3.08) circle (3pt);
\filldraw (1.35,1.98) circle (3pt);
\filldraw (1.03,2) circle (3pt);
\filldraw (-1.15,2.15) circle (3pt);
\filldraw (-1.2,1.85) circle (3pt);
\end{scope}
\begin{scope}[line width=2pt,shift={(10,0)} ]
\draw(-1.2,2).. controls (-0.8,3.23) and (0.8,3.3) .. (1.2,2);
\draw (0.12,0.075).. controls (1,0.5) and (1.4,1) .. (1.2,2);
\draw (0,-3).. controls (-2,-2.5) and (-1.48,-0.47) .. (-0.13,-0.06);
\draw(0,-3) .. controls (4,-4) and (3.9,1.56).. (1.35,1.98);
\draw (1.07,2.02) .. controls (0.7,2.1) and (-0.7,2.1).. (-1.07,2.02);
\draw(-0.15,-3.04) .. controls (-4,-3.8) and (-4,1.6).. (-1.25,2);
\draw (0.15,-2.95).. controls (1.9,-2.5) and (1.5,-0.5) .. (0,0);
\draw (0,0).. controls (-1,0.5) and (-1.4,1) .. (-1.2,2);
\end{scope}
\begin{scope} [white,shift={(10,0)}]
\filldraw (0.15,0.1) circle (3pt);
\filldraw (-0.15,-0.083) circle (3pt);
\filldraw (0.24,-2.89) circle (3pt);
\filldraw (-0.25,-3.08) circle (3pt);
\filldraw (1.35,1.98) circle (3pt);
\filldraw (1.03,2) circle (3pt);
\filldraw (-1.35,1.98) circle (3pt);
\filldraw (-1.03,2) circle (3pt);
\end{scope}


%
%
%
%
%
%
%
%

\end{tikzpicture}}
\caption{On the left hand side we have a diagram of the figure-eight knot. By performing a crossing exchange on the top left crossing, we obtain the unknot diagram on the right hand side.}
\label{fig:croexc}
\end{figure}
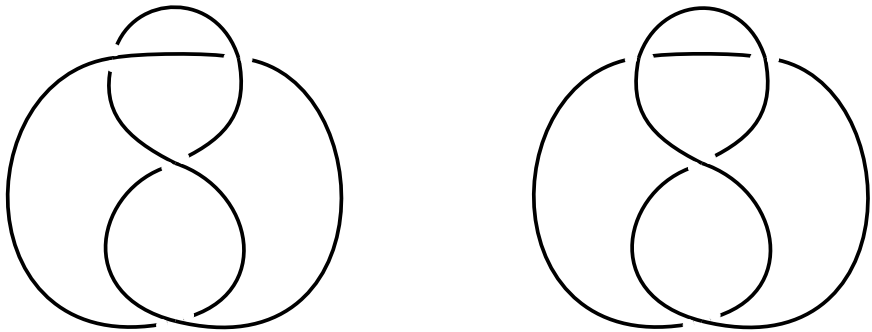

For our arguments, it will be convenient to formulate Theorem~\ref{thm:maindia} in terms of shadows. We recall that the {\em shadow} $S$ of a diagram $D$ is the $4$-regular plane graph obtained by transforming each crossing of $D$ into a vertex. 

Conversely, we may regard a diagram as a shadow for which we provide the over/under-crossing information at each vertex. That is, for each vertex it is prescribed which strand goes locally above (or, equivalently, is the {\em overpass}) and which strand goes locally below (or, equivalently, is the {\em underpass}). We call this information the {\em prescription} at the vertex. 

At each vertex we have two possible prescriptions (corresponding to which of the strands is an overpass at the corresponding crossing). The collection of prescriptions for all the vertices yields an {\em assignment} on $S$, resulting in a diagram $D$. We say that $S$ {\em is the shadow of $D$}, and that $D$ is {\em a diagram of $S$}. If $S$ has $n$ vertices, then there are $2^n$ possible assignments on $S$, and so there are $2^n$ distinct diagrams of $S$. These $2^n$ distinct diagrams are the diagrams that can be obtained by performing crossing exchanges on $D$.

In view of this discussion, it follows that Theorem~\ref{thm:maindia} can be alternatively written as follows.

\begin{theorem}[Equivalent to Theorem~\ref{thm:maindia}]\label{thm:maintheorem}
Let $S$ be a shadow with $n$ vertices. Then $S$ has at least $2^{\sqrt[3]{n}}$ distinct unknot diagrams.
\end{theorem}

We use $\uu(S)$ to denote the set of unknot diagrams of $S$. Thus, with this notation, Theorem~\ref{thm:maintheorem} claims that every shadow $S$ satisfies $|\uu(S)|\ge 2^{\sqrt[3]{n}}$.

Obviously, the exact number of unknot diagrams depends on the shadow under consideration. It is easy to show that out of the $2^4$ diagrams obtained from the shadow of the diagrams in Figure~\ref{fig:croexc}, exactly $12$ are unknot. There are also extreme examples such as the ones illustrated in Figure~\ref{fig:chorizo}: all $2^4$ diagrams of each of these shadows are unknot. It is straightforward to generalize these examples, showing that for every positive integer $n$ there is a shadow on $n$ vertices, all of whose diagrams are unknot.

\begin{figure}[ht!]
\centering
\scalebox{0.7}{\begin{tikzpicture}[line width =1 pt]

\begin{scope}
\end{scope}

\begin{scope}[scale=0.7]
\draw(0,2) .. controls (-1.5,2) and (-1,-1) .. (0,-2);
\draw(0,-2) .. controls (-2,-2) and (-2,3) .. (0,4);
\draw (0,4) .. controls (-3,4) and (-3,-3) .. (0,-4);
\draw (0,-4) .. controls (-4,-4) and (-4,5) .. (0,6);
\draw(0,6) .. controls (-5,6) and (-5,-5) .. (0,-6);
\draw (0,-6) .. controls (-6,-6) and (-9,8) .. (0,8);
\draw (0,2) .. controls (1.5,2) and (1,-1) .. (0,-2);
\draw (0,-2) .. controls (2,-2) and (2,3) .. (0,4);
\draw (0,4) .. controls (3,4) and (3,-3) .. (0,-4);
\draw (0,-4) .. controls (4,-4) and (4,5) .. (0,6);
\draw (0,6) .. controls (5,6) and (5,-5.) .. (0,-6);
\draw (0,-6) .. controls (6,-6) and (9,8) .. (0,8);
\filldraw (0,-2) circle (7pt);
\filldraw (0,-4) circle (7pt);
\filldraw (0,-6) circle (7pt);
\filldraw (0,4) circle (7pt);
\filldraw (0,6) circle (7pt);

\end{scope}

\begin{scope}[shift ={(0,-7)}, scale=0.7]
\filldraw (2,0) circle (7pt);
\filldraw (-2,0) circle (7pt);
\filldraw (6,0) circle (7pt);
\filldraw (-6,0) circle (7pt);

\draw(6,0) .. controls (11,4) and (11,-4) .. (6,0);
\draw(-6,0) .. controls (-11,4) and (-11,-4) .. (-6,0);
\draw(-6,0) .. controls (-5,1.5) and (-3,1.5) .. (-2,0);
\draw(-6,0) .. controls (-5,-1.5) and (-3,-1.5) .. (-2,0);
\draw(-2,0) .. controls (-1,-1.5) and (1,-1.5) .. (2,0);
\draw(-2,0) .. controls (-1,1.5) and (1,1.5) .. (2,0);
\draw(2,0) .. controls (3,-1.5) and (5,-1.5) .. (6,0);
\draw(2,0) .. controls (3,1.5) and (5,1.5) .. (6,0);
\end{scope}

\draw (0,-0.8) node {\LARGE$v$};

\end{tikzpicture}}
\vglue -1 cm
\caption{All the diagrams of these shadows are unknot.}
\label{fig:chorizo}
\end{figure}
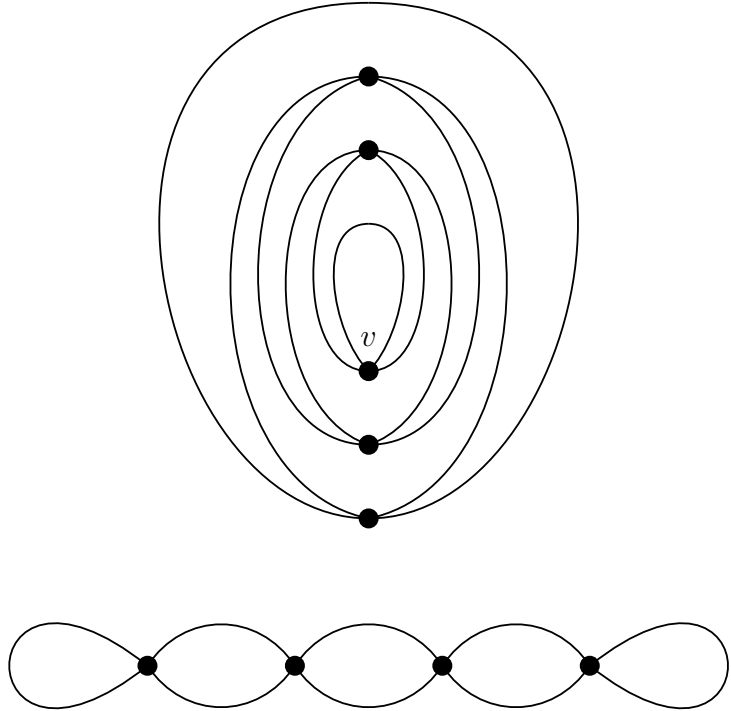

\subsection{Shadows that yield the trefoil knot}

We also investigate which knots (other than the unknot) can be guaranteed to be obtainable from a given shadow. That is, given a knot $K$ and a shadow $S$, is there an assignment on $S$ that gives a diagram of $K$?

The first natural step is the trefoil knot. If $D$ is a diagram whose corresponding knot is equivalent to the trefoil knot, then we say that $D$ is a {\em trefoil diagram}. 

For this particular knot we can give a complete answer. First let us note that there exist arbitrarily large shadows that do not have trefoil diagrams: it suffices to generalize the examples in Figure~\ref{fig:chorizo} to an arbitrary number of vertices. What we will show is that such shadows (that do not have trefoil diagrams) can be easily characterized.

A glaring feature of the shadows in Figure~\ref{fig:chorizo} is that each of their vertices is a cut-vertex. We use Tutte's notion of a cut-vertex~\cite{tut}, which we review in Section~\ref{sec:trefoil}. We show that the property that every vertex is a cut-vertex completely characterizes which shadows have only unknot diagrams, and that every shadow that does not have this feature has a trefoil diagram:

\begin{theorem}\label{thm:trefoil}
Let $S$ be a shadow. Then every diagram of $S$ is unknot if and only if every vertex of $S$ is a cut-vertex. Moreover, if not every vertex of $S$ is a cut-vertex, then $S$ has a trefoil diagram.
\end{theorem}

It is natural to ask how far we can get along the lines of Theorem~\ref{thm:trefoil}. The next obvious step is to investigate the figure-eight knot, which is the only knot with crossing number four. One might think that for the figure-eight knot, a result analogous to Theorem~\ref{thm:trefoil} could hold, at least for all sufficiently large non-simple shadows. 

Now it is easy to construct an arbitrarily large non-simple shadow, none of whose assignments gives a diagram of the figure-eight knot. Indeed, it suffices to artificially grow the trefoil diagram by gluing to it the shadow at the bottom of Figure~\ref{fig:chorizo}, as shown in Figure~\ref{fig:paste}.

\begin{figure}[ht!]
\centering
\scalebox{0.5}{\begin{tikzpicture}[line width=1pt]
\begin{scope}
\filldraw (-6,0) circle (7pt);
\filldraw (-2,0) circle (7pt);
\filldraw (2,0) circle (7pt);
\filldraw (6,0) circle (7pt);
\draw (-10,0) ..controls (-9,1) and  (-7,1) .. (-6,0);
\draw (-10,0) ..controls (-9,-1) and  (-7,-1) .. (-6,0);
\draw (-6,0) ..controls (-5,1) and  (-3,1) .. (-2,0);
\draw (-6,0) ..controls (-5,-1) and  (-3,-1) .. (-2,0);
\draw (-2,0) ..controls (-1,1) and  (1,1) .. (2,0);
\draw (-2,0) ..controls (-1,-1) and  (1,-1) .. (2,0);
\draw (2,0) ..controls (3,1) and  (5,1) .. (6,0);
\draw (2,0) ..controls (3,-1) and  (5,-1) .. (6,0);
\draw (6,0) ..controls (10,3) and  (10,-3) .. (6,0);
\end{scope}

\begin{scope}[shift={(-13.8,-1.6)},scale=2.5]
\draw(-0.1,0.05).. controls  (0,0) and (0,-0).. (0.1,-0.05);
\draw(0.1,0.05).. controls  (0,0.) and (0,-0).. (-0.1,-0.05);
\draw (-0.1,	2.2) .. controls (-0.05,2.21) and  (0.05,2.21) ..(0.1,2.2);
\draw(0.1,2.2).. controls  (1,2.2) and (1.3,0.7).. (0.1,0.05);
\draw(-0.1,2.2).. controls  (-1,2.2) and (-1.3,0.7).. (-0.1,0.05);
\draw (0.1,-0.05) .. controls (1.5,-0.6) and (2.5,1.6).. (0,1.6);
\draw (-0.1,-0.05) .. controls (-1.5,-0.6) and (-2.5,1.6).. (-0,1.6);
\filldraw (0,0) circle (3pt);
\filldraw (0.87,1.45) circle (3pt);
\filldraw (-0.87,1.45) circle (3pt);
\filldraw (1.5,0.65) circle (3pt);
\end{scope}

\end{tikzpicture}}
\caption{Every diagram of this shadow is either a trefoil diagram or an unknot diagram, and this can be easily generalized to obtain arbitrarily large shadows with the same property.}
\label{fig:paste}
\end{figure}
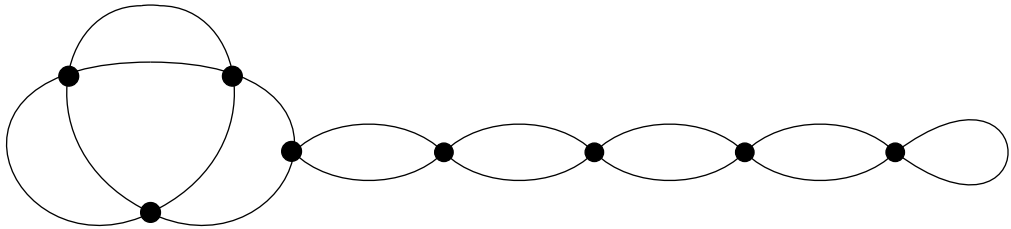

Thus, in order to meaningfully explore a possible extension of Theorem~\ref{thm:trefoil} to shadows other than the trefoil (say, to the figure-eight knot) one must exclude from consideration shadows that have cut vertices; that is, we must focus our attention on $2$-connected shadows.

As it happens, no analogue of Theorem~\ref{thm:trefoil} holds for the figure-eight knot, even if the attention is restricted to $2$-connected shadows, no matter how large. Moreover, no such analogue exists for any knot with positive even crossing number:

\begin{observation}\label{obs:figure8}
Let $K$ be any knot with positive even crossing number, and let $n$ be a positive integer. Then there exists a $2$-connected shadow $S$ with more than $n$ vertices such that, for every diagram $D$ of $S$, the knot corresponding to $D$ is not equivalent to $K$. 
\end{observation}

\subsection{Outline of the proof of Theorem~\ref{thm:maintheorem} and structure of this paper}

In the next section we overview some previous results in the literature related to Theorem~\ref{thm:maintheorem}, and in Section~\ref{sec:shadows} we review some elementary facts on shadows that we use extensively in this paper.

Our proof of Theorem~\ref{thm:maintheorem} is of an inductive nature. In a nutshell the idea is that, given a shadow $S$, we obtain from $S$ a smaller shadow $S'$ that is (loosely speaking) contained in $S$, and then show that every unknot diagram of $S'$ can be extended to at least two unknot diagrams of $S$. Applying a similar procedure to $S'$, by an inductive reasoning we find the claimed lower bound for $\uu(S)$.

For an illustration of one of the techniques we use, we refer the reader to Figure~\ref{fig:4from1}. If we remove from $S$ the edges of the thick cycle $C$ and suppress all resulting degree $2$ vertices, we obtain the shadow $S'$ shown below $S$. We will use the notation $S'=S\cy C$ to indicate that $S'$ is obtained from $S$ in this way. To the right of $S'$ we show an unknot diagram $D'$ of $S'$, and above $D'$ we show how to extend $D'$ to four unknot diagrams of $S$. The idea here is simply that it suffices to have the strand corresponding to $C$ lie entirely above or entirely below the rest of the diagram. This property implies that $D'$ and $D$ are equivalent; since $D'$ is unknot, then these four diagrams are unknot.

We note that this technique cannot be applied using an arbitrary cycle $C$. We can use it only if $C$ is a {\em straight-ahead} cycle, that is, as we traverse $C$ (starting from some of its vertices), then, as we reach an intermediate vertex $v$, we must continue traversing $C$ using the edge opposite to the one from which we arrived to $v$. (Straight-ahead cycles are called {loops} in~\cite{erickson}).

In order to simplify our arguments, we will use the notion of a {\em macro move} in a diagram (following the terminology used by Kaufmann in~\cite{macro}). An instance of a macro move is precisely when we have a strand all of whose crossings lie ``above'' the rest of the diagram; such a strand is an {\em overstrand}. Similarly, if all the crossings of the strand lie ``below'' the rest of the diagram, then the strand is an {\em understrand}. An overstrand (respectively, understrand) can be replaced with another overstrand (respectively, understrand) with the same endpoints, if these endpoints are distinct; the diagram thus obtained is equivalent to the original diagram. If an overstrand or understrand starts and ends at the same crossing point, then we can collapse the strand to this point, and the diagram thus obtained is equivalent to the original diagram. This is precisely the operation performed on the thick strands in diagrams $D_1,D_2,D_3$, and $D_4$ in Figure~\ref{fig:4from1}. We will review macro moves in Section~\ref{sec:gen}.

Back to the discussion on the proof of Theorem~\ref{thm:maintheorem}, we illustrate the second main technique we use in Figure~\ref{fig:2from1}. In the shadow $S$ we identify two parallel edges with common endpoints $u$ and $v$; these two edges and their endpoints form a {\em digon}. We obtain $S'$ by splitting the vertices $u$ and $v$ as shown. Chang and Erickson call this type of operation a $2\to 0$ homotopy move, and this suggests that $S'$ has two fewer vertices than $S$. To the right of $S'$ we have an unknot diagram $D'$ of $S'$. This diagram can be extended to the two illustrated diagrams $D_1$ and $D_2$ of $S$. In each of $D_1$ and $D_2$ we can apply a Reidemeister move of Type II on the strands corresponding to the identified digon, obtaining $D'$. Since $D'$ is unknot, it follows that both $D_1$ and $D_2$ are unknot. 

In Section~\ref{sec:first} we formally explain and develop this first technique we have outlined. The limitation of the first technique is that, in order to successfully apply it recursively, it must be used on straight-ahead cycles that have a relatively small number of vertices: if the cycle whose edges we remove has a large number of vertices, then the resulting shadow $S'$ will have too few vertices, and so the inductive argument will yield a poor lower bound for $|\uu(S)|$. The problem is that we cannot guarantee that there always exist (relatively) small straight-ahead cycles. It is easy to give examples of arbitrarily large shadows for which there are only two cycles on which this technique can be used, and getting rid of one of these cycles trivializes the whole shadow.

For such cases in which the first technique does not suffice, we have at hand the second technique. Now the obvious limitation of this second technique is that there exist shadows that do not have any digons. To handle this possibility we need a variant of this second technique, which can be applied to a pair of internally disjoint paths with common endpoints. Such a pair of paths has a natural underlying digon-like structure, and most of the work needed to prove Theorem~\ref{thm:maintheorem} involves the extension of the second technique to take advantage of such digon-like structures. Part of this work involves a detailed understanding on the existence and structure of digons in a system of two simple closed curves in the plane. This is carried out in Section~\ref{sec:digons}. With this information at our disposal, we then formally lay out the second technique in Section~\ref{sec:second}. 

The final step, given in Section~\ref{sec:proofmain}, consists then on showing how these two techniques can be combined to yield Theorem~\ref{thm:maintheorem}. 

The proof of Theorem~\ref{thm:trefoil} is in Section~\ref{sec:trefoil}. In a nutshell, to prove Theorem~\ref{thm:trefoil} we investigate the structure of a shadow not all of whose vertices are cut vertices, and then explicitly construct a trefoil diagram for any such shadow. Also in this section we describe the infinite family of shadows that proves Observation~\ref{obs:figure8}.
Finally, in Section~\ref{sec:concludingremarks} we present some concluding remarks and open questions.

\section{Related work}\label{sec:background}

Let us start by recalling the well-known argument showing that if $S$ is a shadow, then there is an assignment on $S$ that gives an unknot diagram. Let $S$ be a shadow, and let $p$ be a point in the middle of an edge of $S$. Starting at $p$, we traverse the curve underlying $S$, following any of the two possible directions (see Figure~\ref{fig:stun}(a)). We obtain a diagram from $S$ by following the rule that each time we arrive to a vertex for the first time, then we assign a crossing to this vertex so that the strand we are currently traversing is the overstrand. If we apply this procedure to the shadow in Figure~\ref{fig:stun}(a), we get the unknot diagram shown in Figure~\ref{fig:stun}(b). It is easy to convince oneself that a diagram produced in this manner is unknot; a formal proof can be found in~\cite{adams}*{Section 3.1}.

An unknot diagram obtained in this way is called a {\em standard unknot diagram} (see~\cite{knotdiagrammatics}), and also a {\em descending diagram} (see~\cite{mro} and ~\cite{ozawa}). To understand the motivation behind the terminology ``descending'', let us consider a knot that gets projected to such a diagram. Suppose that the $z$-axis is the one orthogonal to the plane where the diagram is drawn. Then the knot can be realized (in $3$-dimensional space) so that, as we traverse the knot starting at the point $p'$ that projects to $p$, the $z$-coordinates are descending until we get to a point $q'$ whose projection is very close to $p$, and finally we join $q'$ to $p'$ with a $z$-ascending strand.

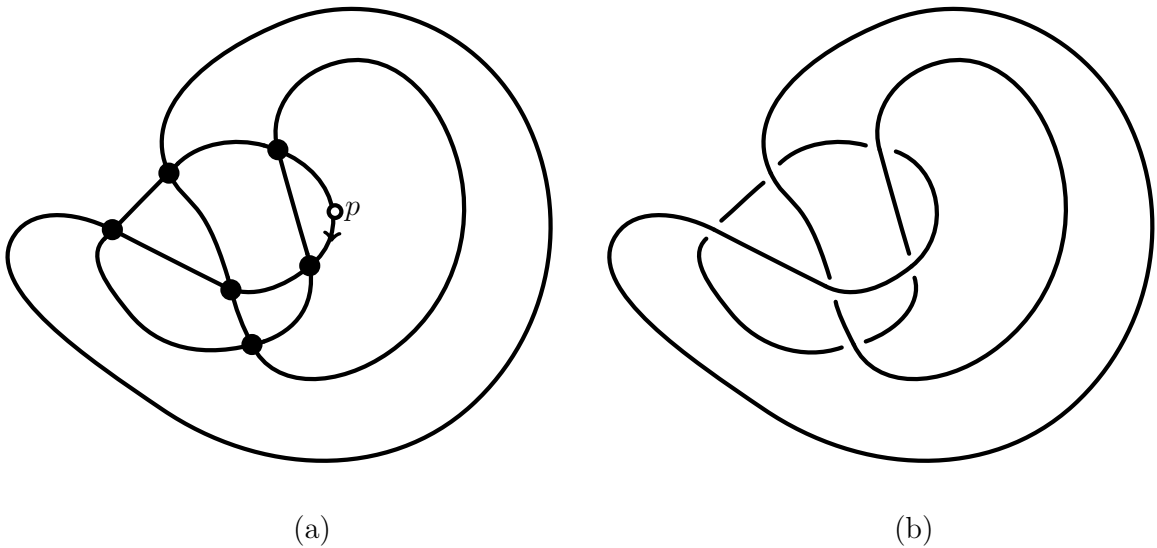
\begin{figure}[ht!]
\centering
\scalebox{0.8}{\begin{tikzpicture}[line width=2.0, line cap=round, line join=round]
  \begin{scope}
    \draw (2.18, 4.30) .. controls (1.54, 4.62) and (0.77, 4.68) .. 
          (0.50, 4.12) .. controls (0.10, 3.29) and (1.74, 2.17) .. 
          (3.01, 1.31) .. controls (4.62, 0.22) and (6.75, 0.10) .. 
          (8.15, 1.38) .. controls (9.47, 2.59) and (9.84, 4.52) .. 
          (9.05, 6.12) .. controls (8.31, 7.61) and (6.59, 8.37) .. 
          (5.06, 7.76) .. controls (3.77, 7.25) and (2.62, 6.29) .. (3.12, 5.24);
    \draw (3.12, 5.24) .. controls (3.40, 4.66) and (3.68, 4.9) .. (4.1, 3.5);
    \draw (4.1, 3.5) .. controls (4.23, 2.94) and (4.36, 2.66) .. (4.50, 2.39);
    \draw (4.50, 2.39) .. controls (4.87, 1.61) and (5.96, 1.68) .. 
          (6.78, 2.26) .. controls (7.78, 2.96) and (8.24, 4.21) .. 
          (7.93, 5.39) .. controls (7.67, 6.40) and (6.94, 7.27) .. 
          (6.01, 7.10) .. controls (5.27, 6.96) and (4.74, 6.31) .. (4.93, 5.63);
    \draw (4.93, 5.63) .. controls (5.09, 5.04) and (5.26, 4.44) .. (5.42, 3.9);
    \draw (5.42, 3.9) .. controls (5.65, 3.04) and (5.19, 2.61) .. (4.7, 2.44);
    \draw (4.7, 2.44) .. controls (3.66, 2.12) and (2.92, 2.35) .. 
          (2.47, 2.90) .. controls (2.11, 3.34) and (1.73, 3.85) .. (2.05, 4.15);
    \draw (2.05, 4.15) .. controls (2.53, 4.65) and (2.77, 4.89) .. (3.27, 5.4);
    \draw (3.27, 5.4) .. controls (3.63, 5.75) and (4.23, 5.83) .. (4.7, 5.7);
    \draw (4.7, 5.7) .. controls (5.93, 5.35) and (6.16, 4.25) .. (5.46, 3.69);
    \draw (5.46, 3.69) .. controls (5.04, 3.35) and (4.50, 3.12) .. (4.03, 3.36);
    \draw (4.03, 3.36) .. controls (3.41, 3.67) and (2.80, 3.98) .. (2.18, 4.30);
\filldraw (2.18,4.30)  circle [radius=4pt];
\filldraw (3.12, 5.24)  circle [radius=4pt];
\filldraw (4.15, 3.3)  circle [radius=4pt];
\filldraw (4.5, 2.39)  circle [radius=4pt];
\filldraw (4.93, 5.63)  circle [radius=4pt];
\filldraw (5.46,3.7)  circle [radius=4pt];
\filldraw (5.88,4.6)  circle [radius=3pt] [black,fill=white] node [black,right=0.2] {\Large  $p$};
\draw [->] (5.83,4.2)-- (5.82,4.1);
\node at (5.5,-0.7) {\Large (a)};
  \end{scope}
  \begin{scope}[shift={(10,0)}]
    \draw (2.18, 4.30) .. controls (1.54, 4.62) and (0.77, 4.68) .. 
          (0.50, 4.12) .. controls (0.10, 3.29) and (1.74, 2.17) .. 
          (3.01, 1.31) .. controls (4.62, 0.22) and (6.75, 0.10) .. 
          (8.15, 1.38) .. controls (9.47, 2.59) and (9.84, 4.52) .. 
          (9.05, 6.12) .. controls (8.31, 7.61) and (6.59, 8.37) .. 
          (5.06, 7.76) .. controls (3.77, 7.25) and (2.62, 6.29) .. (3.12, 5.24);
    \draw (3.12, 5.24) .. controls (3.40, 4.66) and (3.68, 4.9) .. (4.1, 3.5);
    \draw (4.2, 3.1) .. controls (4.23, 2.94) and (4.36, 2.66) .. (4.50, 2.39);
    \draw (4.50, 2.39) .. controls (4.87, 1.61) and (5.96, 1.68) .. 
          (6.78, 2.26) .. controls (7.78, 2.96) and (8.24, 4.21) .. 
          (7.93, 5.39) .. controls (7.67, 6.40) and (6.94, 7.27) .. 
          (6.01, 7.10) .. controls (5.27, 6.96) and (4.74, 6.31) .. (4.93, 5.63);
    \draw (4.93, 5.63) .. controls (5.09, 5.04) and (5.26, 4.44) .. (5.42, 3.9);
    \draw (5.51, 3.5) .. controls (5.65, 3.04) and (5.19, 2.61) .. (4.7, 2.44);
    \draw (4.3, 2.34) .. controls (3.66, 2.12) and (2.92, 2.35) .. 
          (2.47, 2.90) .. controls (2.11, 3.34) and (1.73, 3.85) .. (2.05, 4.15);
    \draw (2.3, 4.45) .. controls (2.53, 4.65) and (2.77, 4.89) .. (3, 5.08);
    \draw (3.27, 5.4) .. controls (3.63, 5.75) and (4.23, 5.83) .. (4.7, 5.7);
    \draw (5.2, 5.6) .. controls (5.93, 5.35) and (6.16, 4.25) .. (5.46, 3.69);
    \draw (5.46, 3.69) .. controls (5.04, 3.35) and (4.50, 3.12) .. (4.03, 3.36);
    \draw (4.03, 3.36) .. controls (3.41, 3.67) and (2.80, 3.98) .. (2.18, 4.30);
\node at (5.5,-0.7) {\Large (b)};
  \end{scope}
\end{tikzpicture}}
\caption{Creating an uknot diagram starting from any given shadow.}
\label{fig:stun}
\end{figure}

As we mentioned in Section~\ref{sec:intro}, for any given shadow $S$ with $n$ vertices, the number of unknot diagrams of $S$ that can be produced with this construction is at most $4n$. Indeed, there are $2n$ ways to choose $p$ (since $S$ has $2n$ edges), and for each of these $2n$ choices there are two possible directions to follow.

There seem to be very few results in the literature involving the number of unknot diagrams of an arbitrary shadow. Recently, Cantarella et al.~\cite{canta} carried out a complete investigation of the knot types that arise in all shadows with $10$ or fewer crossings. Their results show that, among all the diagrams associated to this huge collection of shadows, roughly $78\%$ are unknot diagrams. Cantarella et al.~observed that this large proportion is explained by the existence of ``tree-like'' shadows, such as the ones depicted in Figure~\ref{fig:chorizo}. These are shadows in which every vertex is a cut-vertex, and so an easy induction shows that all of their diagrams are unknot.


There is a collection of quite interesting, deep results in the literature related to the complementary aspect of Theorem~\ref{thm:maintheorem}: out of the $2^n$ diagrams that have $S$ as their shadow, how many are {\em not} unknot diagrams?

According to Sumners and Whittington~\cite{sumners1}, this was first asked in the context of long linear polymer chains, independently by Frisch and Wasserman~\cite{frisch} and Delbruck~\cite{delbruck}. In our context, the Frisch-Wasserman-Delbruck Conjecture may be roughly paraphrased as follows: if $p(n)$ is the probability that a random knot diagram with $n$ crossings is knotted, then $p(n)\to 1$ as $n\to\infty$.

The first issue to settle in an investigation of the Frisch-Wasserman-Delbruck Conjecture is: what is a random knot diagram? In~\cite{sumners1}, Sumners and Whittington investigated (and settled in the affirmative) this conjecture in the model of self-avoiding walks on the three-dimensional simple cubic lattice (see~\cite{pippenger} for a closely related result). The conjecture has also been settled in other models of space curves, such as self-avoiding Gaussian polygons and self-avoiding equilateral polygons~\cites{arsuaga,buck,jungreis,diao,diao2}. 

The problem of proposing suitable models of random knot diagrams is of interest by itself. The associated difficulties are explained and discussed in~\cite{diao3}, where two different such models are presented and investigated. We also refer the reader to the preliminary report by Dunfield et al.~in~\cite{dunfield}. It is also worth mentioning the very recent work of Even-Zohar et al.~\cite{even-zohar}, where several rigorous results are established on the distributions of knot and link invariants for the Petaluma model, which is based on the representation of knots and links as petal diagrams~\cite{adamspetal}. In a very interesting recent development, Chapman~\cites{chapman1,chapman2} proved the Frisch-Wasserman-Delbruck conjecture under a very general and natural model.

From the work of Chapman we know that if we take a random diagram on $n$ vertices, then the probability that it is an unknot diagram decreases exponentially with $n$. Our work in this paper focuses on the abundance of unknot diagrams associated to an arbitrary (that is, not to a random) shadow.

\section{Basic facts on shadows}\label{sec:shadows}

A shadow of a diagram with at least one crossing is a $4$-regular plane graph. For our upcoming arguments it will be convenient to admit a vertex-less simple closed curve as a shadow. This {\em trivial} shadow is the shadow of a {\em trivial} diagram, that is, a diagram with no crossings.

Thus every shadow is a $4$-regular plane graph, but not every $4$-regular plane graph is a shadow of a knot. A $4$-regular graph is a shadow if and only if it has a straight-ahead Eulerian closed walk. We recall that a {\em straight-ahead} walk in a $4$-regular graph is a walk in which every time we reach a vertex $v$ in the walk, the next edge in the walk is the opposite edge to the one from which we arrived to $v$. A walk is {\em closed} if its final vertex is the same as its initial vertex, and it is {\em Eulerian} if each edge of the graph gets traversed exactly once in the graph.

We also recall that if $C$ is a connected subgraph of a graph $G$, and every vertex of $C$ has degree two (in $C$), then $C$ is a {\em cycle} of $G$. For our purposes, a vertex-less graph (which, we recall, consists of a single edge, homeomorphic to a simple closed curve) will be considered a cycle.

If $C$ is a cycle of a shadow $S$, and there exists a straight-ahead walk that traverses each edge of $C$ exactly once, then $C$ is a {\em straight-ahead} cycle of $S$. The initial (and final) vertex of this walk is a {\em root} of $C$. For an illustration of a straight-ahead cycle in a shadow we refer the reader to the top left of Figure~\ref{fig:4from1}. The trivial shadow has no vertices, but by convention we regard it as a straight-ahead cycle. Finally we note that it is easy to see that a straight-ahead cycle in a nontrivial shadow has a unique root. 

In our arguments we make extensive use of the following properties of shadows. The first one is totally straightforward.

\begin{fact}\label{fac:fact1}
Let $v$ be a vertex of a shadow $S$. Then there is an Eulerian straight-ahead closed walk $W$ that starts and ends at $v$. The walk $W$ gets naturally decomposed into two edge-disjoint straight-ahead closed walks $W_1,W_2$, both of which have $v$ as their startvertex and endvertex.
\end{fact}


\begin{fact}\label{fac:fact2}
Every nontrivial shadow has at least two distinct (necessarily edge-disjoint) straight-ahead cycles.
\end{fact}

\begin{proof}
Let $v$ be a vertex of a shadow $S$, and let $W_1,W_2$ be the straight-ahead closed walks (from Fact~\ref{fac:fact1}) that start and end at $v$. We show that as we traverse $W_1$, we must encounter a straight-ahead cycle. Since the argument is symmetric on $W_1$ and $W_2$, and these walks are edge-disjoint, this proves the fact.

We prove this by induction on the number of vertices of the walk. If $v$ is the only vertex of $W_1$, then this is obvious, as $W_1$ is itself a straight-ahead cycle. Thus suppose that $W_1$ has at least one vertex other than $v$. Suppose that $W_1$ is not a straight-ahead cycle. Then there exists a vertex $u$ other than $v$ that gets visited twice as we traverse $W_1$. The subwalk of $W_1$ that starts and ends at $u$ has fewer vertices that $W_1$, and by the induction hypothesis it has a straight-ahead cycle $C$, which is also a straight-ahead cycle of $W_1$.
\end{proof}

\section{Macro moves}\label{sec:gen}

The three Reidemeister moves are particular cases of more general operations that turn a diagram $D$ into an {\em equivalent} diagram $D'$, that is, a diagram $D'$ whose corresponding knot is isotopic to the knot represented by $D$. Following Kauffman~\cite{macro}, we call these operations {\em macro moves}. We quote the following lively passage from~\cite{macro}: 

\bigskip

{\leftskip=20pt\rightskip=20pt
\noindent\cite{macro}*{Section 2}{\em ``In a macro move, we identify an arc that passes entirely under some piece of the diagram (or entirely over) and shift this part of the arc, keeping it under (or over) during the shift.''} 
\par}

\bigskip

We also refer the reader to the interesting discussion on macro moves (although not using this terminology) in {\tt mathoverflow.net} initiated by Gowers~\cite{gowers}.

We will use two particular kinds of macro moves. Before describing them, let us clarify the notion of strand that we will use throughout this work.

\begin{figure}[htbp]
    \centering
        \scalebox{0.55}{ \begin{tikzpicture}[line width=1pt]
\begin{scope}[shift={(-3.,-0.5)}]
\draw (-4,-3) ..controls (-3,-3) and (-1,-2).. (0,0)..controls (1,1.5) and (-8,2.5).. (-2,3);
\filldraw[white] (0,0) circle (10pt); 
\end{scope}
\begin{scope}[shift={(-3,-0.5)}]
\draw(-2,3)..controls (-0.5,7) and (5,-1.5).. (0,0) ..controls (-5,1.8) and (-7,-3) .. (-4,-3);
\filldraw[white] (-1.9,3.2) circle (5pt); 
\draw (-2.3,3.4) node { \huge$x$};
\draw (0.5,-2) node {\huge $\sigma$};
\end{scope}
\begin{scope}[shift={(-3,-0.5)}]
\filldraw[white] (1.15,2.68) circle (18pt); 
\filldraw[white] (-4.3,-0.2) circle (12pt); 
\filldraw[white] (-1.8,-2.1) circle (15pt); 
\filldraw[white] (-3.6,2.15) circle (10pt); 
\end{scope}
\begin{scope}[shift={(-3,-0.5)}]
\draw [line width =3pt] (-2,3)..controls  (5,4) and (3,-3) .. (-3,-2).. controls (-4,-1.8) and (-4.2,-1).. (-4.3,-0.5);
\draw[line width =3pt](-4.3,-0.5) .. controls (-4.4,0) and (-4.3,1.2).. (-4.2,1.5) .. controls (-4,3) and (-2,1.6) .. (-2,2.67)   ;
\draw[,->, line width=3pt]  (-2.4,2.23) --(-2.6,2.23);
\end{scope}
\begin{scope}[shift={(7,-0.5)}]
\draw (-4,-3) ..controls (-3,-3) and (-1,-2).. (0,0)..controls (1,1.5) and (-8,2.5).. (-2,3);
\filldraw[white] (0,0) circle (10pt); 
\end{scope}
\begin{scope}[shift={(7,-0.5)}]
\draw(-2,3)..controls (-0.5,7) and (5,-1.5).. (0,0) ..controls (-5,1.8) and (-7,-3) .. (-4,-3);
\filldraw[black] (-2,3) circle (2pt);
\end{scope}
\end{tikzpicture}}
    \caption{A macro move of Type A. In the diagram $D$ on the left-hand side, the thick strand $\sigma$ is an overstrand based at $x$. The diagram on the right-hand side is obtained by collapsing $\sigma$ to $x$, that is, by removing $\sigma\setminus\{x\}$ from $D$. Therefore these diagrams are equivalent.}
    \label{fig:typeb}
\end{figure}
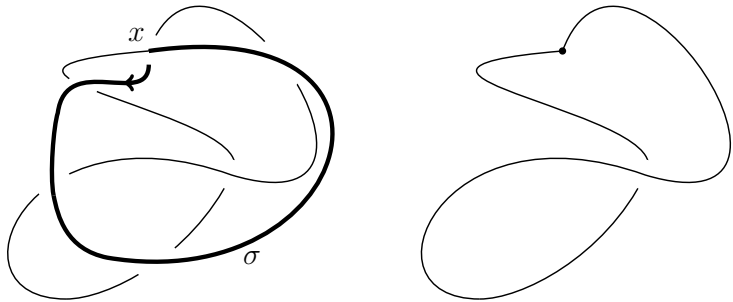

We adopt the viewpoint that a {\em strand} of a knot $K$ is a connected (in our context, it will always be proper) subset of $K$. If a diagram $D$ is obtained as a projection of $K$, then we refer to the part of $D$ that corresponds to a strand of $K$ also as a {\em strand} of $D$. We actually only use strands in diagrams, so there should be no confusion, whenever we work with a strand, as to whether we refer to a strand as a part of a diagram or as a piece of a knot. As an illustration, we refer the reader to Figures~\ref{fig:typeb} and~\ref{fig:typea}. In the former, we have a strand $\sigma$ whose endpoints coincide, in a crossing point, and in the latter we illustrate strands $\sigma,\sigma'$ with distinct endpoints.

The first macro move we use is illustrated in Figure~\ref{fig:typeb}. Let $x$ be a crossing point in a diagram $D$, and let $\sigma$ be a proper strand that starts and ends at $x$. Further suppose that $\sigma$ has no self-crossings other than $x$. If as we traverse $\sigma$ starting from $x$ (in any of the two possible directions), every crossing we find is an overpass (respectively, underpass) for $\sigma$, then $\sigma$ is an {\em overstrand} (respectively, {\em understrand}) {\em based at} $x$.

Informally speaking, if $\sigma$ is an overstrand (respectively, understrand) of $D$ based at $x$, then $\sigma$ lies above (respectively, below) the rest of the diagram. In either case we can then collapse $\sigma$ to the point $x$, and the diagram $D'$ thus obtained is equivalent to $D$. We say that $D'$ is obtained from $D$ by a {\em macro move of Type A}. We note that the classical Reidemeister move of Type I is a particular instance of a macro move of Type A, namely in the case in which $x$ is the only crossing of $\sigma$.

The other macro move we use is illustrated in Figure~\ref{fig:typea}. Let $x,y$ be distinct non-crossing points in a diagram $D$. Suppose that $\sigma$ is a strand of $D$ with endpoints $x,y$, that does not cross itself, and has the property that as we traverse $\sigma$, all the crossings  we find are overpasses (respectively, underpasses) for $\sigma$. Then we say that $\sigma$ is an {\em overstrand} (respectively, {\em understrand}) of $D$. Let $D'$ be a diagram obtained from $D$ by removing $\sigma$, and replacing it by another strand also with endpoints $x,y$, that is an overstrand (respectively, understrand) in $D'$. Then $D$ and $D'$ are clearly equivalent diagrams. We say that $D'$ is obtained from $D$ by a {\em macro move of Type B.} 

\begin{figure}[htbp]
    \centering
        \scalebox{0.45}{\begin{tikzpicture}[line width=2pt]
\begin{scope}[shift={(0,0.5)}]
\draw (-3,4) ..controls (0,2) and (0,-2).. (3,-4);
\filldraw[white] (0,0) circle (12pt);
\draw (3,4) ..controls (0,2) and (0,-2).. (-3,-4);
\filldraw[white] (0.7,1.5) circle (10pt);
\filldraw[white] (-0.7,1.5) circle (10pt);
\filldraw[white] (0.95,-1.75) circle (10pt);
\filldraw[white] (-0.95,-1.75) circle (10pt);
\end{scope}
\filldraw (-4,2) circle (7pt) node[shift={(-0.8,0)}]{\Huge$x$};
\filldraw (4,2) circle (7pt) node[shift={(0.8,0)}]{\Huge$y$};
\draw (-4,2) .. controls (-4.2,2.7) and (-4.5,3.5) .. (-5,4.5);
\draw (4,2) .. controls (4.2,2.7) and (4.5,3.5) .. (5,4.5);
\draw[line width=5pt] (-4,2).. controls (-3,-2.5)and (3,-2.5).. (4,2);
\draw  (-3.7,-0.1) .. controls (-4.5,-1.) and (-4.4,-0.8) .. (-6.2,-3.5);
\draw  (3.7,-0.1) .. controls (4.5,-1) and (4.4,-0.8) .. (6.2,-3.5);
\draw  (-2.8,0.8) .. controls (-1,2.5) and (1,2.5) .. (2.8,0.8);
\draw (2.6,-1.2 ) node {\Huge$\sigma$};

\begin{scope}[shift={(17,0.5)}]
\draw (-3,4) ..controls (0,2) and (0,-2).. (3,-4);
\filldraw[white] (0,0) circle (12pt);
\draw (3,4) ..controls (0,2) and (0,-2).. (-3,-4);
\filldraw[white] (0.7,1.5) circle (10pt);
\filldraw[white] (-0.7,1.5) circle (10pt);
\filldraw[white] (1.6,2.8) circle (10pt);
\filldraw[white] (-1.6,2.8) circle (10pt);
\end{scope}
\begin{scope}[shift={(17,0)}]
\filldraw (-4,2) circle (7pt) node[shift={(-0.8,-0.5)}]{\Huge$x$};
\filldraw (4,2) circle (7pt) node[shift={(0.8,-0.5)}]{\Huge$y$};
\draw (-4,2) .. controls (-4.2,2.7) and (-4.5,3.5) .. (-5,4.5);
\draw (4,2) .. controls (4.2,2.7) and (4.5,3.5) .. (5,4.5);
\draw[line width=5pt] (-4,2).. controls (-3,3.95)and (3,3.95).. (4,2);
\draw  (-3.7,-0.1) .. controls (-4.5,-1.) and (-4.4,-0.8) .. (-6.2,-3.5);
\draw  (3.7,-0.1) .. controls (4.5,-1) and (4.4,-0.8) .. (6.2,-3.5);
\draw  (-3.7,-0.1) .. controls (-1,2.8) and (1,2.8) .. (3.7,-0.1);
\draw (0,4.1 ) node {\Huge$\sigma'$};
\end{scope}

\end{tikzpicture}}
    \caption{A macro move of Type B. The thick strand $\sigma$ with endpoints $x,y$ is an overstrand in the left-hand side diagram. We obtain the right-hand side diagram by replacing this overstrand with the overstrand $\sigma'$. These diagrams are thus equivalent.
}
    \label{fig:typea}
\end{figure}
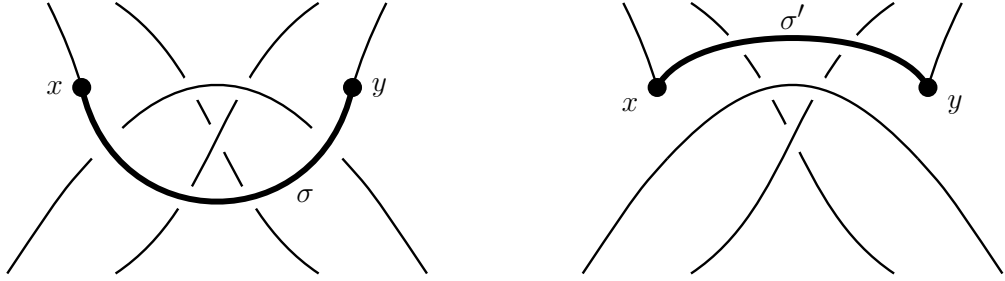

We note that the classical Reidemeister moves of Types II and III are particular instances of macro moves of Type B.

\section{Obtaining unknot diagrams using cycle decompositions}\label{sec:first}

We now formally lay out the first tool we develop to recursively construct unknot diagrams of a shadow. Let $S$ be a shadow, and suppose that $C$ is a straight-ahead cycle of $S$. Let $S'$ be the shadow obtained from $S$ by removing the edges of $C$, and suppressing any resulting degree $2$ vertices. We use the notation $S'=S\cy C$ to denote that $S'$ is obtained from $S$ in this way.

The first key claim is that if $S'=S\cy C$, then $|\uu(S)|\ge 2|\uu(S')|$, and that a stronger inequality holds if $C$ has more than one vertex:

\begin{proposition}\label{prop:w1}
Let $S$ be a shadow, and let $C$ be a straight-ahead cycle of $S$. Let $S'=S\cy C$. Then $|\uu(S)|\ge 2|\uu(S')|$. Moreover, if $C$ has more than one vertex, then $|\uu(S)|\ge 4|\uu(S')|$.
\end{proposition}

\begin{proof}
The main idea is that an unknot diagram of $S'$ can be ``extended'' to at least two unknot diagrams of $S$ (and to at least four, if $C$ has more than one vertex). Suppose first that $C$ has more than one vertex, and let $c$ be the root of $C$. We refer the reader to Figure~\ref{fig:4from1} for an illustration of the proof, where $C$ is the cycle formed by the thick edges. 


Starting from any unknot diagram $D'$ of $S'$, we obtain four distinct unknot diagrams of $S$ as follows. First, we let each vertex in $S$ that is in $S'$ maintain the prescription it has in $D'$. Then the prescriptions at the vertices in $S\setminus (S'\cup\{c\})$ are given so that in the resulting diagram of $S$, the strand corresponding to $C$ is either an overstrand or an understrand (two possibilities). Moreover, we still have the freedom to assign any of the two possible prescriptions at $c$. This yields in total four distinct diagrams of $S$. Each of these four diagrams is unknot: one can apply a macro move of Type A to the strand corresponding to $C$ (this is either an over- or an under-strand, so such a move can be performed), to obtain the unknot diagram $D'$. Thus these four diagrams are unknot. 

Finally, if $c$ is the only vertex of $C$, then we extend $D'$ to two unknot diagrams of $S$, by giving to $c$ each of the two possible prescriptions. 
\end{proof}

\begin{figure}[ht!]
\centering
\hglue -0.5 cm\scalebox{0.8}{\begin{tikzpicture}[, line cap=round, line join=round]

\begin{scope}[shift={(-8,1)}]

\draw(-0.1,0.05).. controls  (0,0) and (0,-0).. (0.1,-0.05);
\draw(0.1,0.05).. controls  (0,0.) and (0,-0).. (-0.1,-0.05);
\draw[black,shift={(0,2.15)}, line width=3](-0.1,0.05).. controls  (0,0.05) and (0,-0).. (0.1,-0.1);
\draw[black,shift={(0,2.15)}, line width=3](0.1,0.05).. controls  (0,0.05) and (0,0).. (-0.1,-0.1);
\filldraw[white] (0,2.25) circle (5pt);
\draw(0.1,2.2).. controls  (1,2.2) and (1.3,0.7).. (0.1,0.05);
\draw(-0.1,2.2).. controls  (-1,2.2) and (-1.3,0.7).. (-0.1,0.05);
\draw (0.1,-0.05) .. controls (1.5,-0.6) and (2.5,1.6).. (0,1.6);
\draw (-0.1,-0.05) .. controls (-1.5,-0.6) and (-2.5,1.6).. (-0,1.6);
\draw[black, line width=3] (0,-0.3).. controls  (1.3,-0.3) and (0.8,1.).. (0.1,2.05);
\draw[black, line width=3] (0,-0.3).. controls  (-1.3,-0.3) and (-0.8,1.).. (-0.1,2.05);
\filldraw[black] (0,2.16) circle (3pt);
%
\filldraw (0.38,1.57) circle (3pt);
\filldraw (-0.38,1.57) circle (3pt);
\filldraw (-.75,0.71) circle (3pt);
\filldraw (.75,0.71) circle (3pt);
\filldraw (.62,-0.15) circle (3pt);
\filldraw (-.62,-0.15) circle (3pt);
\filldraw (0,0) circle (3pt);
\filldraw (0.87,1.45) circle (3pt);
\filldraw (-0.87,1.45) circle (3pt);
\draw (0,-0.8) node {\Large$S$};
\end{scope}
\begin{scope}[shift={(-8,-3.5)}]
\draw(-0.1,0.05).. controls  (0,0) and (0,-0).. (0.1,-0.05);
\draw(0.1,0.05).. controls  (0,0.) and (0,-0).. (-0.1,-0.05);
\draw (-0.1,	2.2) .. controls (-0.05,2.21) and  (0.05,2.21) ..(0.1,2.2);
\draw(0.1,2.2).. controls  (1,2.2) and (1.3,0.7).. (0.1,0.05);
\draw(-0.1,2.2).. controls  (-1,2.2) and (-1.3,0.7).. (-0.1,0.05);
\draw (0.1,-0.05) .. controls (1.5,-0.6) and (2.5,1.6).. (0,1.6);
\draw (-0.1,-0.05) .. controls (-1.5,-0.6) and (-2.5,1.6).. (-0,1.6);
\filldraw (0,0) circle (3pt);
\filldraw (0.87,1.45) circle (3pt);
\filldraw (-0.87,1.45) circle (3pt);
\draw (0,-0.6) node {\Large$S'$};
\end{scope}
\begin{scope}[shift={(3,-3.5)}]
\draw(-0.1,0.05).. controls  (0,0) and (0,-0).. (0.1,-0.05);
\draw (-0.1,	2.2) .. controls (-0.05,2.21) and  (0.05,2.21) ..(0.1,2.2);
\draw(0.1,2.2).. controls  (1,2.2) and (1.3,0.7).. (0.1,0.05);
\draw(-0.1,2.2).. controls  (-1,2.2) and (-1.3,0.7).. (-0.1,0.05);
\filldraw[white] (0.87,1.45) circle (5pt);
\filldraw[white] (-0.87,1.45) circle (5pt);
\draw (0.1,-0.05) .. controls (1.5,-0.6) and (2.5,1.6).. (0,1.6);
\draw (-0.1,-0.05) .. controls (-1.5,-0.6) and (-2.5,1.6).. (-0,1.6);
\draw (0,-0.6) node {\Large$D'$};
\end{scope}
\begin{scope}[shift={(9,1)}]
\draw(-0.1,0.05).. controls  (0,0) and (0,-0).. (0.1,-0.05);
\draw(0.1,2.2).. controls  (1,2.2) and (1.3,0.7).. (0.1,0.05);
\draw(-0.1,2.2).. controls  (-1,2.2) and (-1.3,0.7).. (-0.1,0.05);
\filldraw[white] (0.87,1.45) circle (5pt);
\filldraw[white] (-0.87,1.45) circle (5pt);
black,shift={(0,-2)}](0.2,0.2).. controls  (0,0.1) and (0,0.1).. (-0.2,0.2);
\draw (0.1,-0.05) .. controls (1.5,-0.6) and (2.5,1.6).. (0,1.6);
\draw (-0.1,-0.05) .. controls (-1.5,-0.6) and (-2.5,1.6).. (-0,1.6);
\filldraw[white] (0.38,1.57) circle (5pt);
\filldraw[white] (-0.38,1.57) circle (5pt);
\filldraw[white] (-.75,0.71) circle (5pt);
\filldraw[white] (.75,0.71) circle (5pt);
\filldraw[white] (.62,-0.15) circle (6pt);
\filldraw[white] (-.62,-0.15) circle (6pt);
\draw[black, line width=3] (0,-0.3).. controls  (1.3,-0.3) and (0.8,1.).. (0.1,2.05);
\draw[black, line width=3] (0,-0.3).. controls  (-1.3,-0.3) and (-0.8,1.).. (-0.1,2.05);
\draw[black,shift={(0,2.15)}, line width=3](0.1,0.05).. controls  (0,0.05) and (0,0).. (-0.1,-0.1);
\draw (0,-0.8) node {\Large$D_4$};
\end{scope}
\begin{scope}[shift={(5,1)}]
\draw(-0.1,0.05).. controls  (0,0) and (0,-0).. (0.1,-0.05);
\draw(0.1,2.2).. controls  (1,2.2) and (1.3,0.7).. (0.1,0.05);
\draw(-0.1,2.2).. controls  (-1,2.2) and (-1.3,0.7).. (-0.1,0.05);
\filldraw[white] (0.87,1.45) circle (5pt);
\filldraw[white] (-0.87,1.45) circle (5pt);
\draw (0.1,-0.05) .. controls (1.5,-0.6) and (2.5,1.6).. (0,1.6);
\draw (-0.1,-0.05) .. controls (-1.5,-0.6) and (-2.5,1.6).. (-0,1.6);
\filldraw[white] (0.38,1.57) circle (5pt);
\filldraw[white] (-0.38,1.57) circle (5pt);
\filldraw[white] (-.75,0.71) circle (5pt);
\filldraw[white] (.75,0.71) circle (5pt);
\filldraw[white] (.62,-0.15) circle (5pt);
\filldraw[white] (-.62,-0.15) circle (5pt);
\draw[black, line width=3] (0,-0.3).. controls  (1.3,-0.3) and (0.8,1.).. (0.1,2.05);
\draw[black, line width=3] (0,-0.3).. controls  (-1.3,-0.3) and (-0.8,1.).. (-0.1,2.05);
\draw[black,shift={(0,2.15)}, line width=3](-0.1,0.05).. controls  (0,0.05) and (0,-0).. (0.1,-0.1);
\draw (0,-0.8) node {\Large$D_3$};
\end{scope}
\begin{scope}[shift={(1,1)}]
\draw[black, line width=3] (0,-0.3).. controls  (1.3,-0.3) and (0.8,1.).. (0.1,2.05);
\draw[black, line width=3] (0,-0.3).. controls  (-1.3,-0.3) and (-0.8,1.).. (-0.1,2.05);
\draw[black,shift={(0,2.15)}, line width=3](0.1,0.05).. controls  (0,0.05) and (0,0).. (-0.1,-0.1);

\filldraw[white] (0.38,1.57) circle (5pt);
\filldraw[white] (-0.38,1.57) circle (5pt);

\filldraw[white] (-.75,0.71) circle (5pt);
\filldraw[white] (.75,0.71) circle (5pt);
\filldraw[white] (.62,-0.15) circle (5pt);
\filldraw[white] (-.62,-0.15) circle (5pt);
\draw(-0.1,0.05).. controls  (0,0) and (0,-0).. (0.1,-0.05);
\draw(0.1,2.2).. controls  (1,2.2) and (1.3,0.7).. (0.1,0.05);
\draw(-0.1,2.2).. controls  (-1,2.2) and (-1.3,0.7).. (-0.1,0.05);
\filldraw[white] (0.87,1.45) circle (5pt);
\filldraw[white] (-0.87,1.45) circle (5pt);
black,shift={(0,-2)}](0.2,0.2).. controls  (0,0.1) and (0,0.1).. (-0.2,0.2);
\draw (0.1,-0.05) .. controls (1.5,-0.6) and (2.5,1.6).. (0,1.6);
\draw (-0.1,-0.05) .. controls (-1.5,-0.6) and (-2.5,1.6).. (-0,1.6);
\draw[black,shift={(0,2.15)}, line width=3](0.1,0.05).. controls  (0,0.05) and (0,0).. (-0.1,-0.1);
\draw (0,-0.8) node {\Large$D_2$};
\end{scope}
\begin{scope}[shift={(-3,1)}]
\draw[black, line width=3] (0,-0.3).. controls  (1.3,-0.3) and (0.8,1.).. (0.1,2.05);
\draw[black, line width=3] (0,-0.3).. controls  (-1.3,-0.3) and (-0.8,1.).. (-0.1,2.05);

\filldraw[white] (0.38,1.57) circle (5pt);
\filldraw[white] (-0.38,1.57) circle (5pt);
\filldraw[white] (-.75,0.71) circle (5pt);
\filldraw[white] (.75,0.71) circle (5pt);
\filldraw[white] (.62,-0.15) circle (5pt);
\filldraw[white] (-.62,-0.15) circle (5pt);
\draw(-0.1,0.05).. controls  (0,0) and (0,-0).. (0.1,-0.05);
\draw(0.1,2.2).. controls  (1,2.2) and (1.3,0.7).. (0.1,0.05);
\draw(-0.1,2.2).. controls  (-1,2.2) and (-1.3,0.7).. (-0.1,0.05);
\filldraw[white] (0.87,1.45) circle (5pt);
\filldraw[white] (-0.87,1.45) circle (5pt);
black,shift={(0,-2)}](0.2,0.2).. controls  (0,0.1) and (0,0.1).. (-0.2,0.2);
\draw (0.1,-0.05) .. controls (1.5,-0.6) and (2.5,1.6).. (0,1.6);
\draw (-0.1,-0.05) .. controls (-1.5,-0.6) and (-2.5,1.6).. (-0,1.6);
\draw[black,shift={(0,2.15)}, line width=3](-0.1,0.05).. controls  (0,0.05) and (0,-0).. (0.1,-0.1);
\draw (0,-0.8) node {$D_1$};
\end{scope}
\begin{scope}[gray,shift={(0,-0.4)}]
\draw(-3.5,-0.2) ..controls (-3,-1) and (2,-0.1) ..(3,-0.6);
\draw(9.5,-0.2) ..controls (9,-1) and (4,-0.1) ..(3,-0.6);
\end{scope}
\end{tikzpicture}}
\caption{Illustration of the proof of Proposition~\ref{prop:w1}.}
\label{fig:4from1}
\end{figure}
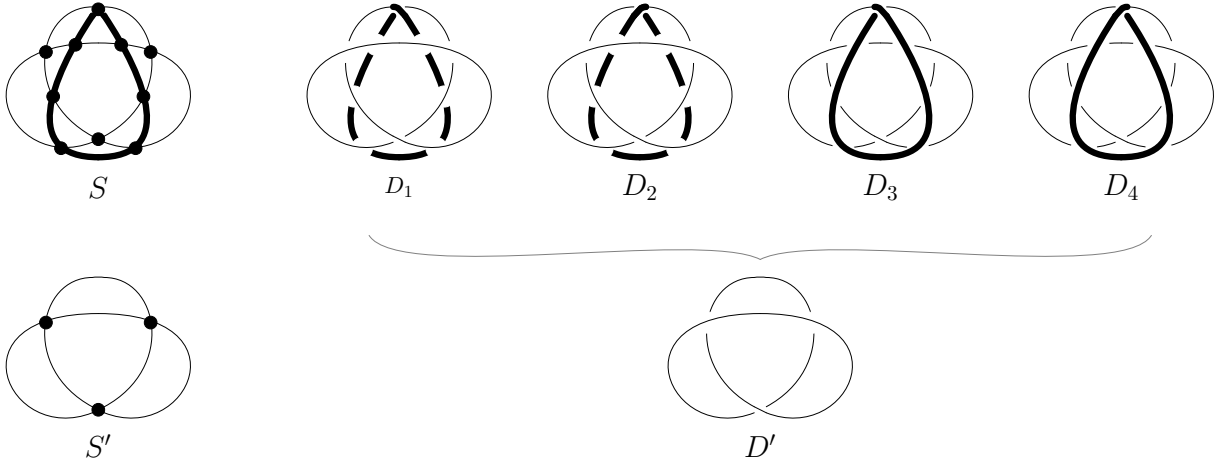

We wish to iteratively apply this proposition until we end up with a trivial shadow. The formal setting makes use of the following definition.

\bigskip
\noindent{\bf Definition. (Cycle decomposition).}
Let $S$ be a shadow. Suppose that $S=S_1,S_2,$ $\ldots,S_p$ is a sequence of shadows, with the following properties:
\begin{itemize}
\item For $i=1,2,\ldots,p-1$, there is a straight-ahead cycle $C_i$ of $S_i$ such that $S_{i+1}=S_i\cy C_i$.
\item $S_p$ is the trivial shadow.
\end{itemize}

Set $C_p:=S_p$. Then the sequence $C_1,C_2,\ldots,C_{p}$ is a {\em cycle decomposition} of $S$ {\em of size $p$}.
\bigskip

In Figure~\ref{fig:t1-1} we give an example of a cycle decomposition. We remark that in a cycle decomposition $C_1,C_2,\ldots,C_p$ of a shadow $S$, the cycles $C_2,\ldots,C_p$ are not necessarily straight-ahead cycles in $S$. Moreover, they do not even need be cycles of $S$. On the other hand, for $i=1,2,\ldots,p$ there is a cycle $D_i$ in $S$ such that $D_i$ is a subdivision of $C_i$. The cycles $D_1,D_2,\ldots,D_p$ are pairwise edge-disjoint, and $S=\cup_{i=1}^p D_i$. The sequence $D_1,D_2,\ldots,D_p$ is the {\em primary sequence} of $C_1,C_2,\ldots,C_p$.

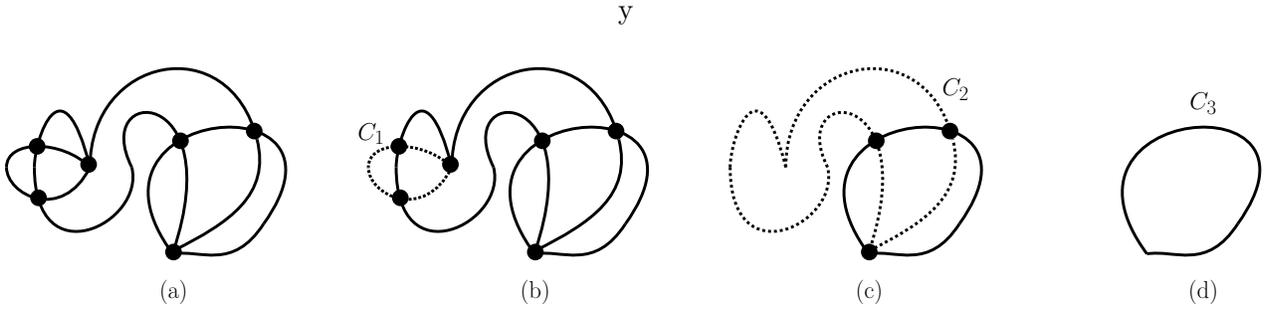
\begin{figure}[ht!]
    \centering
y        \scalebox{0.37}{ \begin{tikzpicture}[line width=3pt]

\begin{scope}[rotate=90,shift={(0,0)}]
\draw (0,3) .. controls (1,4)  and (1,6).. (0,6); 
\draw (0,3) .. controls (-2,4)  and (-1,6).. (0,6); 
\draw(0,3) .. controls (4,4)  and (1,5).. (0,5); 
\draw  (0,5) .. controls (-4,5)  and (-2,1).. (0,1.5); 
\draw  (0,1.5) .. controls (3,3)  and (3,-2).. (-3,-0); 
\draw (0,3) .. controls (4,3)  and (5,-2).. (1,-3); 
\draw(-3,0) .. controls (-2,-2)  and (-1,-3.5).. (1,-3); 
\draw [shift={(-0.8,-1.6)}, rotate=190] (2,-2) .. controls (-4,-5)  and (-3,6).. (1.5,1.4); 
\draw [shift={(-0.8,-1.6)}, rotate=190](2,-2) .. controls (2,-1)  and (3,0).. (1.5,1.4); 
\filldraw (0.1,3.05) circle (7pt);
\filldraw (-3.05,0) circle (7pt);
\end{scope}
\begin{scope}
\filldraw (2.9,1.3) circle (7pt);
\filldraw (0.25,0.95) circle (7pt);
\filldraw (-4.9,0.75) circle (7pt);
\filldraw (-4.85,-1.1) circle (7pt);
\draw (0,-4.5) node  {\Huge (a)};
\end{scope}
\begin{scope}[shift={(13,0)},rotate=90]
\draw[dotted] (0,3) .. controls (1,4)  and (1,6).. (0,6); 
\draw[dotted] (0,3) .. controls (-2,4)  and (-1,6).. (0,6); 
\draw(0,3) .. controls (4,4)  and (1,5).. (0,5); 
\draw  (0,5) .. controls (-4,5)  and (-2,1).. (0,1.5); 
\draw  (0,1.5) .. controls (3,3)  and (3,-2).. (-3,-0); 
\draw (0,3) .. controls (4,3)  and (5,-2).. (1,-3); 
\draw(-3,0) .. controls (-2,-2)  and (-1,-3.5).. (1,-3); 
\draw [shift={(-0.8,-1.6)}, rotate=190] (2,-2) .. controls (-4,-5)  and (-3,6).. (1.5,1.4); 
\draw [shift={(-0.8,-1.6)}, rotate=190](2,-2) .. controls (2,-1)  and (3,0).. (1.5,1.4); 
\filldraw (0.1,3.05) circle (7pt);
\filldraw (-3.05,0) circle (7pt);
\end{scope}
\begin{scope}[shift={(13,0)}]
\filldraw (2.9,1.3) circle (7pt);
\filldraw (0.25,0.95) circle (7pt);
\filldraw (-4.9,0.75) circle (7pt);
\filldraw (-4.85,-1.1) circle (7pt);
\draw (-5.9,1.2) node  {\Huge $C_1$};
\draw (0,-4.5) node  {\Huge (b)};

\end{scope}

\begin{scope}[shift={(25,0)},rotate=90]
\draw[dashed](0,3) .. controls (4,4)  and (1,5).. (0,5); 
\draw[dashed]  (0,5) .. controls (-4,5)  and (-2,1).. (0,1.5); 
\draw [dashed] (0,1.5) .. controls (3,3)  and (3,-2).. (-3,-0); 
\draw [dashed](0,3) .. controls (4,3)  and (5,-2).. (1,-3); 
\draw[dashed](-3,0) .. controls (-2,-2)  and (-1,-3.5).. (1,-3); 
\draw [shift={(-0.8,-1.6)}, rotate=190] (2,-2) .. controls (-4,-5)  and (-3,6).. (1.5,1.4); 
\draw [shift={(-0.8,-1.6)}, rotate=190](2,-2) .. controls (2,-1)  and (3,0).. (1.5,1.4); 
\filldraw (-3.05,0) circle (7pt);
\end{scope}
\begin{scope}[shift={(25,0)}]
\filldraw (2.9,1.3) circle (7pt);
\filldraw (0.25,0.95) circle (7pt);
\draw (3.1,2.8) node  {\Huge $C_2$};
\draw (0,-4.5) node  {\Huge (c)};
\end{scope}

\begin{scope}[shift={(35,0)},rotate=90]
\draw [shift={(-0.8,-1.6)}, rotate=190] (2,-2) .. controls (-4,-5)  and (-3,6).. (1.5,1.4); 
\draw [shift={(-0.8,-1.6)}, rotate=190](2,-2) .. controls (2,-1)  and (3,0).. (1.5,1.4); 
\end{scope}
\begin{scope}[shift={(37,0)}]
\draw (0,2.3) node  {\Huge $C_3$};
\draw (0,-4.5) node  {\Huge (d)};

\end{scope}

\end{tikzpicture}}
    \caption{A cycle decomposition of the shadow $S=S_1$ in (a). In (b) we highlight with dotted edges a straight-ahead cycle $C_1$ of $S_1$. After removing the edges of $C_1$, and suppressing the three resulting degree $2$ vertices, we reach the shadow $S_2=S_1\cy C_1$ in (c). The dotted cycle $C_2$ in (c) is a straight-ahead cycle of $S_2$, and in (d) we show the shadow $S_3=S_2\cy C_2$. We finally set $C_3=S_3$ (thus $C_3$ is a vertex-free cycle). The sequence $C_1,C_2,C_3$ is then a cycle decomposition of $S=S_1$.}
    \label{fig:t1-1}
\end{figure}


We can now give the first main ingredient in the proof of Theorem~\ref{thm:maintheorem}.

\begin{lemma}\label{lem:ifone}
Let $S$ be a shadow with $n$ vertices. Suppose that $S$ has a cycle decomposition of size at least ${\sqrt[3]{n}}$. Then $|\uu(S)| \ge 2^{\sqrt[3]{n}}$.
\end{lemma}

\begin{proof}
Let $S$ be a shadow with $n$ vertices, and let $S_1,S_2,\ldots,S_p$ be a sequence of shadows obtained from a cycle decomposition of $S$, where $p\ge \sqrt[3]{n}$. A repeated application of Proposition~\ref{prop:w1} shows that $|\uu(S)|\ge 2^{p-1}|\uu(S_p)|$. Since $|\uu(S_p)|=1$ ($S_p$ is trivial), then $|\uu(S)| \ge 2^{p-1}$.

The bound $|\uu(S)|\ge 2^{p-1}$ is the best we can obtain from Proposition~\ref{prop:w1} only if at each step in the cycle decomposition, we remove a cycle that has only one vertex. Otherwise we can apply at least once the second part of Proposition~\ref{prop:w1}, and obtain that $|\uu(S)|\ge 2^p$. In such a case we are then done, since $p\ge \sqrt[3]{n}$.

Thus we are done unless, at each step in the cycle decomposition, we remove a cycle that has only one vertex. It is straightforward to see that this can happen only if $p=n+1$. In this case the bound $|\uu(S)|\ge 2^{p-1}$ reads $|\uu(S)|\ge 2^{n}$. Since $2^{n}> 2^{\sqrt[3]{n}}$ for every positive integer $n$, it follows that in this case we are also done.
\end{proof}

This lemma settles Theorem~\ref{thm:maintheorem} if $S$ has a cycle decomposition of the given size. If all cycle decompositions of $S$ are smaller, then we need a different approach. Before we introduce this second tool we establish some elementary results on digons in systems of two simple closed curves.

\section{Digons in systems of two closed curves}\label{sec:digons}

Lemma~\ref{lem:ifone} in the previous section settles Theorem~\ref{thm:maintheorem} for shadows with sufficiently large cycle decompositions. For a shadow $S$ that does not have this property, our approach will be to find two cycles $B,R$ in $S$ that have many vertices in common, and then to use these cycles to produce many unknot diagrams of $S$. This approach relies on the existence of digons (with special properties) in a system of two simple closed curves in the plane. In this section we investigate these systems of curves, and establish the required results.

Let $\beta$ be a simple closed curve in the plane. A {\em segment} of $\beta$ is a subset of $\beta$ that is homeomorphic to the closed interval $[0,1]$. Thus a segment of $\beta$ is simply a non-closed curve (including its endpoints) contained in $\beta$.

Let $\beta,\rho$ be two non-disjoint simple closed curves. A {\em digon of} $(\beta,\rho)$ is a pair of segments $(\alpha,\gamma)$, where $\alpha$ is in $\beta$ and $\gamma$ is in $\rho$,  such that (i) $\alpha$ and $\gamma$ have common endpoints $x,y$; (ii) the only points of $\beta$ that are in $\gamma$ are $x$ and $y$; and (iii) the only points of $\rho$ that are in $\alpha$ are $x$ and $y$. In Figure~\ref{fig:twoc1} we illustrate this concept.

\begin{figure}[ht!]
\centering
\scalebox{0.6}{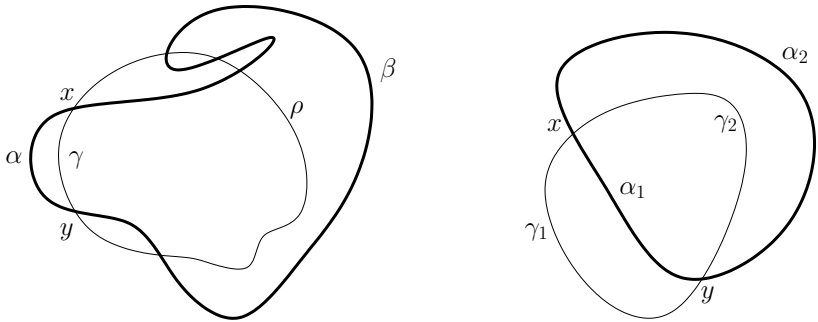}
\caption{On the left hand side we have two simple closed curves $\beta$ (thick) and $\rho$ (thin). The pair $(\alpha,\gamma)$ is a digon of $(\beta,\rho)$. On the right hand side we have two simple closed curves $\beta=\alpha_1\cup\alpha_2$ (thick) and $\rho=\gamma_1\cup\gamma_2$ (thin). This last example illustrates that a digon does not necessarily bound an empty disk (that is, a disk whose interior does not contain any point in the curves). Indeed, the pair $(\alpha_2,\gamma_1)$ is a digon of $(\alpha_1\cup\alpha_2,\gamma_1\cup\gamma_2)$, and $\alpha_2\cup \gamma_1$ does not bound an empty disk. 
}
\label{fig:twoc1}
\end{figure}

If $(\alpha,\gamma)$ is a digon of $(\beta,\rho)$, then $\alpha\cup\gamma$ is a simple closed curve. Properties (ii) and (iii) above imply that at least one of the two connected components of $\real^2\setminus(\alpha\cup\gamma)$ does not have any point in $\beta\cup\rho$ (that is, at least one of these connected components is an {\em empty region}). We remark that it is not necessarily true that this empty region is the disk bounded by $\alpha\cup\gamma$ (see right hand side of Figure~\ref{fig:twoc1}).

We now state and prove the results on digons that we need later. An intersection of two curves at a point $x$ is a {\em crossing} if the curves intersect transversally at $x$. Otherwise the intersection is {\em tangential}.

\begin{proposition}\label{prop:dig1}
Let $\beta,\rho$ be non-disjoint simple closed curves in the plane. Suppose that every intersection point of $\beta$ and $\rho$ is a crossing (rather than tangential). Then there are at least $4$ distinct digons of $(\beta,\rho)$. Moreover, if $b$ is a point of $\beta$ that is not in $\rho$, and $r$ is a point of $\rho$ that is not in $\beta$, then there exists a digon $(\alpha,\gamma)$ of $(\beta,\rho)$ such that $\alpha\cup\gamma$ contains neither $b$ nor $r$.
\end{proposition}

\begin{proof}
To help comprehension, we colour $\beta$ blue and $\rho$ red. Let $\Delta$ denote the closed disk bounded by $\beta$. 

We can naturally regard the part of $\beta\cup\rho$ contained in $\Delta$ as a $3$-regular plane graph: each vertex is incident with two blue edges and one red edge. It is easily seen that if we take the dual of this plane graph, and discard the vertex corresponding to the unbounded face, the result is a tree. Now each leaf of this tree corresponds to a digon of $(\beta,\rho)$. Since each tree with at least two vertices has at least two leaves, then there are at least two digons of $(\beta,\rho)$ whose red segments are inside $\Delta$. A totally analogous argument shows that there are at least two digons of $(\beta,\rho)$ whose red segments are outside $\Delta$. This shows the existence of at least four distinct digons of $(\beta,\rho)$.

The existence of a digon $(\alpha,\gamma)$ of $(\beta,\rho)$ such that $\alpha\cup\gamma$ contains neither $b$ nor $r$ is easily worked out by a simple case analysis if $\beta$ and $\rho$ have exactly two intersection points. We note that, in this case, there is exactly one digon of $(\beta,\rho)$ with this property.

Suppose  $\beta$ and $\rho$ have at least $4$ points in common. Then if $(\alpha,\gamma)$ and $(\alpha',\gamma')$ are distinct digons, we must have $\alpha\neq\alpha'$ and $\gamma\neq\gamma'$. Thus only one digon of $(\beta,\rho)$ can contain $b$, and also only one can contain $r$. As there are at least four digons of $(\beta,\rho)$, it follows that there are at least two (and in particular at least one) digons $(\alpha,\gamma)$ of $(\beta,\rho)$ such that $\alpha\cup\gamma$ contains neither $b$ nor $r$. 
\end{proof}

\begin{proposition}\label{prop:dig2}
Let $\beta,\rho$ be simple closed curves in the plane that intersect in $\ell\ge 3$ points. Suppose that there is exactly one intersection point $z$ of $\beta$ and $\rho$ that is tangential; all the other intersection points are crossings. Then there exists a digon $(\alpha,\gamma)$ of $(\beta,\rho)$ such that $\alpha\cup\gamma$ does not contain $z$.
\end{proposition}

\begin{proof}
As in the proof of Proposition~\ref{prop:dig1}, we colour $\beta$ blue and $\rho$ red, and let $\Delta$ denote the closed disk bounded by $\beta$.

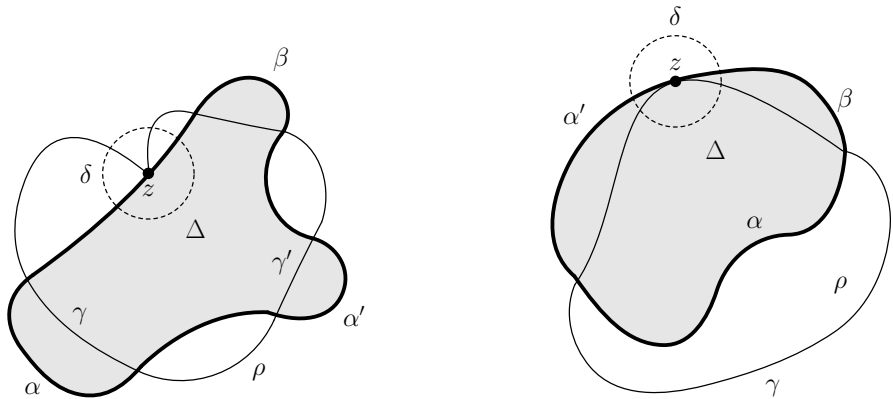
\begin{figure}[ht!]
\centering
\scalebox{0.48}{\begin{tikzpicture}[line width=1., line cap=round, line join=round]
\begin{scope}[fill=black!10,]
    \filldraw  (5.51, 8.03) .. controls (4.37, 6.21) and (2.74, 4.78) .. (1, 3.53) .. controls (0.13, 2.91) and (0.24, 2.10) .. 
          (0.68, 1.50) .. controls (1.12, 0.90) and (1.64, 0.31) .. 
          (2.37, 0.21) .. controls (2.91, 0.13) and (3.43, 0.33) ..  (4.0, 1.00) .. controls (4.98, 1.94) and (6.22, 2.55) .. (7.5, 2.5) .. controls (8.62, 2.11) and (9.40, 2.42) .. 
          (9.61, 3.12) .. controls (9.81, 3.75) and (9.38, 4.40) ..  (8.71, 4.56) .. controls (7.49, 4.86) and (7.05, 6.39) ..  (7.88, 7.43) .. controls (8.25, 7.90) and (8.05, 8.60) .. 
          (7.49, 8.88) .. controls (6.78, 9.24) and (5.97, 8.76) .. (5.51, 8.03) node [below=80]{\huge $\Delta$};
  \end{scope}
\begin{scope}[color=black, line width=3]
    \draw (5.51, 8.03) .. controls (4.37, 6.21) and (2.74, 4.78) .. (1.00, 3.53);
    \draw (1, 3.53) .. controls (0.13, 2.91) and (0.24, 2.10) .. 
          (0.68, 1.50) .. controls (1.12, 0.90) and (1.64, 0.31) .. 
          (2.37, 0.21) .. controls (2.91, 0.13) and (3.43, 0.33) .. (4, 1);
    \draw (4.0, 1.00) .. controls (4.98, 1.94) and (6.22, 2.55) .. (7.5, 2.5);
    \draw (7.5, 2.5) .. controls (8.62, 2.11) and (9.40, 2.42) .. 
          (9.61, 3.12) .. controls (9.81, 3.75) and (9.38, 4.40) .. (8.71, 4.56);
    \draw (8.71, 4.56) .. controls (7.49, 4.86) and (7.05, 6.39) .. (7.88, 7.43)node[above=45] {\huge $\beta$};
    \draw (7.88, 7.43) .. controls (8.25, 7.90) and (8.05, 8.60) .. 
          (7.49, 8.88) .. controls (6.78, 9.24) and (5.97, 8.76) .. (5.51, 8.03);
\node at (2.3,2.4) [black] {\huge $\gamma$} ;
\node at (1,0.4) [black] {\huge $\alpha$} ;
\node at (7.9,3.9) [black] {\huge $\gamma'$} ;
\node at (9.9,2.4) [black] {\huge $\alpha'$} ;
  \end{scope}
  \begin{scope}[color=black]
    \draw (9, 5) .. controls (9.29, 5.83) and (9.07, 7.13) .. (8.0, 7.5);
    \draw (8, 7.5) .. controls (7.03, 7.65) and (6.36, 7.82) .. (5.5, 8);
    \draw (5.5, 8.) .. controls (4.86, 8.20) and (3.95, 8.14) .. 
          (4.23, 6.3) .. controls (2.18, 8.02) and (1.30, 7.31) .. 
          (0.86, 6.34) .. controls (0.44, 5.40) and (0.34, 4.32) .. (0.84, 3.42);
    \draw (0.84, 3.42) .. controls (1.50, 2.21) and (2.70, 1.44) .. (3.94, 0.86) node[right=85] {\huge $\rho$};
    \draw (3.94, 0.86) .. controls (5.36, 0.21) and (7.03, 0.85) .. (7.68, 2.28);
    \draw (7.68, 2.28) .. controls (8.00, 2.98) and (8.31, 3.68) .. (9, 5);
\filldraw (4.2,6.35) [black] circle [radius=4pt]node [below=5] {\huge $z$} ;
\draw (4.2,6.35) [black,dashed] circle [radius=36pt]node[left=40] {\huge$\delta$} ;
  \end{scope}
 \begin{scope}[fill=black!10,shift={(15,0)}]
    \filldraw (8.5, 7.) .. controls (8.42, 7.74) and (8.05, 8.24) .. 
          (7.64, 8.67) .. controls (6.87, 9.47) and (5.55, 9.25) .. 
          (4.39, 9.06) .. controls (2.90, 8.81) and (1.56, 7.96) .. 
          (0.86, 6.63) .. controls (0.30, 5.56) and (0.12, 4.26) .. (1., 3.5) .. controls (1.77, 2.39) and (2.48, 1.61) ..    (3.46, 1.59) .. controls (4.22, 1.57) and (4.69, 2.32) .. 
          (4.94, 3.09) .. controls (5.22, 3.98) and (6.00, 4.63) .. 
          (6.93, 4.65) .. controls (7.90, 4.67) and (8.43, 5.73) .. (8.5, 7)node[left=90] {\huge $\Delta$};
  \end{scope}
  \begin{scope}[line width=3,color=black,shift={(15,0)}]
    \draw (8.5, 7.) .. controls (8.42, 7.74) and (8.05, 8.24) .. 
          (7.64, 8.67) .. controls (6.87, 9.47) and (5.55, 9.25) .. 
          (4.39, 9.06) .. controls (2.90, 8.81) and (1.56, 7.96) .. 
          (0.86, 6.63) .. controls (0.30, 5.56) and (0.12, 4.26) .. (1, 3.5);
    \draw (1., 3.5) .. controls (1.77, 2.39) and (2.48, 1.61) .. 
          (3.46, 1.59) .. controls (4.22, 1.57) and (4.69, 2.32) .. 
          (4.94, 3.09) .. controls (5.22, 3.98) and (6.00, 4.63) .. 
          (6.93, 4.65) .. controls (7.90, 4.67) and (8.43, 5.73) .. (8.5, 7)node[above=25] {\huge $\beta$};
\end{scope}
  \begin{scope}[color=black,shift={(15,0)}]
    \draw (1.01, 3.26) .. controls (0.69, 2.49) and (0.90, 1.61) .. 
          (1.44, 0.98) .. controls (2.19, 0.11) and (3.47, 0.18) .. 
          (4.61, 0.47) .. controls (5.88, 0.79) and (7.13, 1.22) .. 
          (8.21, 1.96) .. controls (9.14, 2.60) and (9.64, 3.67) .. 
          (9.74, 4.80) .. controls (9.83, 5.81) and (9.36, 6.81) .. (8.42, 6.98);
    \draw (8.42, 6.98) .. controls (1.33, 12.) and (3, 6.12) .. (1.01, 3.26)node[right=200] {\huge $\rho$};
\filldraw (3.8,8.9) [black] circle [radius=4pt]node [above=5] {\huge $z$} ;
\draw (3.8,8.9) [black,dashed] circle [radius=36pt]node [above=40]{\huge $\delta$} ;
\node at (6.5,0.4) [black] {\huge $\gamma$} ;
\node at (6,5) [black] {\huge $\alpha$} ;
\node at (1,8) [black] {\huge $\alpha'$} ;
  \end{scope}
\end{tikzpicture}}
\caption{{Illustration of the proof of Proposition~\ref{prop:dig2}. On the left hand side we have the case in which $\delta \cap \rho$ lies outside the disk $\Delta$ bounded by $\beta$. Both $\gamma$ and $\gamma'$ are segments inside $\Delta$, and  both $(\alpha,\gamma)$ and $(\alpha',\gamma')$ are digons, neither of which contains $z$. In the example on the right hand side, $\delta \cap \rho$ lies inside $\Delta$, and $\gamma$ is the only red segment that lies outside $\Delta$. By letting $\alpha$ and $\alpha'$ be the blue segments into which $\rho$ naturally breaks $\beta$, we have that $(\alpha,\gamma)$ and $(\alpha',\gamma)$ are both digons. Of these two digons, only $(\alpha,\gamma)$ does not contain $z$.}}
\label{fig:twoc2}
\end{figure}

Suppose that if we take a small disk $\delta$ centered at $z$ (small enough so that $z$ is the only point in $\beta\cap\rho$ that is in $\delta$), then $\rho\cap\delta$ lies outside $\Delta$ (see Figure~\ref{fig:twoc2}). The other possibility ($\rho\cap\delta$ lies inside $\Delta$) is handled in a totally analogous manner.

Now we proceed as in the proof of Proposition~\ref{prop:dig1}, ignoring the part of $\beta\cup\rho$ outside $\Delta$, and find at least two distinct digons $(\alpha,\gamma),(\alpha',\gamma')$ of $(\beta,\rho)$, such that $\gamma$ and $\gamma'$ (which are not necessarily distinct) are inside $\Delta$. Only one of $\alpha$ and $\alpha'$ can contain $z$. If $\alpha$ contains $z$, then $(\alpha',\gamma')$ is a digon of $(\beta,\rho)$ that does not contain $z$. If $\alpha'$ contains $z$, then $(\alpha,\gamma)$  is a digon of $(\beta,\rho)$ that does not contain $z$.
\end{proof}

\section{Obtaining unknot diagrams using digon-like structures}\label{sec:second}

When we have a shadow $S$ not having a sufficiently large cycle decomposition, and thus for which Lemma~\ref{lem:ifone} does not guarantee enough unknot diagrams, the idea is to find two straight-ahead cycles in $S$ with many common vertices, and then to use these two cycles to find many unknot diagrams of $S$. This approach is summarized in the following statement, our second main tool in the proof of Theorem~\ref{thm:maintheorem}.

\begin{lemma}\label{lem:dum}
Let $S$ be a shadow, and let $B$ and $R$ be distinct straight-ahead cycles of $S$. Suppose that $B$ and $R$ have exactly $m$ vertices in common. Then $|\uu(S)|\ge 2^{m/2}$. 
\end{lemma}

Although the arguments in the proof of this lemma depend on the parity of $m$, the core idea is the same in both cases. We will give separately the proofs for the cases $m$ odd and $m$ even. Even though in principle it would have been possible to present a unified proof, we have decided against this possibility. First, it would only have resulted in a marginal saving of space. Second, dealing with the easier case ($m$ odd) first will allow us to present the main idea in the proof, without the additional complications involved in the case when $m$ is even.

In both cases, to help comprehension we colour the edges of $B$ and $R$ blue and red, respectively.

\begin{proof}[Proof of Lemma~\ref{lem:dum} for odd $m$]
We start by noting that since is $m$ odd, then $S=B\cup R$. Indeed, since $B$ and $R$ are straight-ahead cycles, it follows that if we regard them as simple closed curves, then they can have at most one non-crossing (that is, at most one tangential) intersection. An elementary argument using the Jordan curve theorem shows that the number of crossing intersections is even. Since $m$ is odd, then these curves must have exactly one tangential intersection, which must then be the common root $r$ of $B$ and $R$. If we start a straight-ahead traversal of $S$ starting at $r$, we traverse entirely one of $B$ and $R$, return to $r$, and then traverse the other cycle, coming back again to $r$ having traversed the whole shadow $S$. Thus $S=B\cup R$, as claimed.

The proof is by induction on $m$. The base case $m=1$ is trivial, as in this case the lemma only requires two unknot diagrams, and it is immediately seen that a shadow with one vertex has (exactly) two unknot diagrams.

Suppose then that $m\ge 3$. It follows from Proposition~\ref{prop:dig2} that there exist edges $e,f$ with common endpoints $u,v$, where $e$ is blue and $f$ is red, such that $e\cup f$ is a digon. Moreover, it follows from that proposition that such edges $e,f$ exist so that neither $u$ nor $v$ is $r$.

Let $S'$ be the shadow obtained from $S$ by ``splitting'' each of $u$ and $v$ into two degree $2$ vertices, and then suppressing these four degree $2$ vertices, as illustrated in Figure~\ref{fig:valid1}. It is easy to see that the resulting graph $S'$ is indeed a shadow, with straight-ahead cycles $B',R'$ naturally induced from $B$ and $R$, whose common root is still $r$, and such that $S'=B'\cup R'$.

\begin{figure}[ht!]
\centering
\scalebox{01.0}{ \begin{tikzpicture}
\begin{scope}[line width =2,shift={(-6,0)}]
\draw(-1.5,-1) .. controls (-1.5,1) and (1.5,1) .. (1.5,-1);
\end{scope}
\begin{scope}[,shift={(-6,0)}]
\draw(-1.5,1) .. controls (-1.5,-1) and (1.5,-1) .. (1.5,1);
 \filldraw[black] (-1.15,0) circle (3pt);
 \filldraw[black] (1.15,0) circle (3pt);
\draw[black](-1.8,0)  node { \Large{$u$}};
\draw[black](1.8,0)  node { {\Large$v$}};
\draw[black](0,0.9)  node { {\Large$e$}};
\draw[black](0,-1)  node {{\Large$f$}};
\draw[](-1.5,-1.5)  node {\Large$B$};
\draw[](-1.5,1.5)  node {{\Large$R$}};
\end{scope}

\begin{scope}[,shift={(0,0)}]
\draw[line width=2](-1.28,0.3) .. controls (-1.0,0.8) and (1.0,0.8) .. (1.28,0.3);
\draw[](-1.28,0.3) .. controls (-1.29,0.3) and (-1.70,0.4) .. (-1.7,1);
\draw[](1.28,0.3) .. controls (1.29,0.3) and (1.70,0.4) .. (1.7,1);
\end{scope}
\begin{scope}[,shift={(0,0)}]
\draw(-1.28,-0.3) .. controls (-1.0,-0.8) and (1.0,-0.8) .. (1.28,-0.3);
\draw[line width=2](-1.28,-0.3) .. controls (-1.29,-0.3) and (-1.70,-0.4) .. (-1.7,-1);
\draw[line width=2](1.28,-0.3) .. controls (1.29,-0.3) and (1.70,-0.4) .. (1.7,-1);
 \filldraw[black] (-1.27,.3) circle (3pt);
 \filldraw[black] (-1.27,-.3) circle (3pt);
 \filldraw[black] (1.27,0.3) circle (3pt);
 \filldraw[black] (1.27,-0.3) circle (3pt);
\draw[black](-1.8,0.3)  node {  {\Large$u$}};
\draw[black](-1.8,-0.3)  node {  {\Large$u'$}};
\draw[black](1.9,0.3)  node {  {\Large$v$}};
\draw[black](1.9,-0.3)  node {  {\Large$v'$}};
\draw[black](0,0.9)  node {  {\Large$e$}};
\draw[black](0,-1)  node {  {\Large$f$}};
\end{scope}

\begin{scope}[,shift={(6,0)}]
\draw(-1.28,0.3) .. controls (-1.0,0.8) and (1.0,0.8) .. (1.28,0.3);
\draw(-1.28,0.3) .. controls (-1.29,0.3) and (-1.70,0.4) .. (-1.7,1);
\draw(1.28,0.3) .. controls (1.29,0.3) and (1.70,0.4) .. (1.7,1);
\end{scope}
\begin{scope}[line width=2,shift={(6,0)}]
\draw(-1.3,-0.3) .. controls (-1.0,-0.8) and (1.0,-0.8) .. (1.3,-0.3);
\draw(-1.28,-0.3) .. controls (-1.29,-0.3) and (-1.70,-0.4) .. (-1.7,-1);
\draw(1.28,-0.3) .. controls (1.29,-0.3) and (1.70,-0.4) .. (1.7,-1);
\draw[](-1.5,-1.5)  node {  {\Large$B'$}};
\draw[](-1.5,1.5)  node {  {\Large$R'$}};
\end{scope}

%
%

\end{tikzpicture}}
\caption{ Obtaining $S'$ from $S$ by splitting two vertices. The relevant part of the blue cycle $B$ (respectively, red cycle $R$) is drawn with thick (respectively, thin) edges. The blue edge $e$ and the red edge $f$ have common endpoints $u,v$. We split each of $u$ and $v$ into two degree $2$ vertices, and then suppress these four vertices. The resulting shadow $S'$ also has the property that it is the union of two straight-ahead cycles. We can naturally label these cycles $B'$ and $R'$, letting $B'$ (respectively, $R'$) be the cycle that contains the remains of the red edge $f$ (respectively, the blue edge $e$).
}
\label{fig:valid1}
\end{figure}
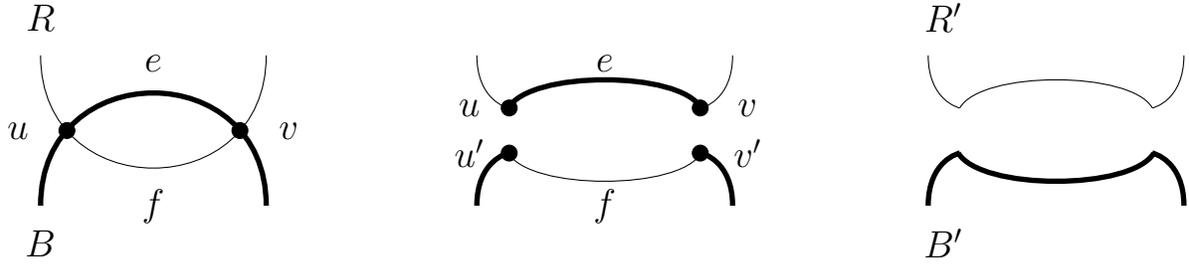

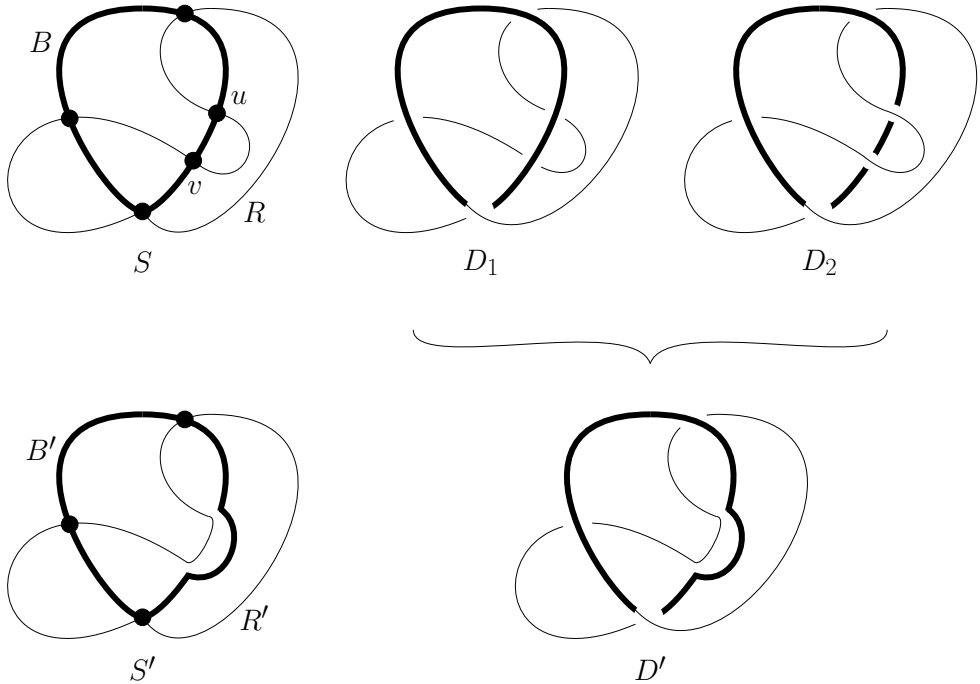
\begin{figure}[ht!]
\centering
\scalebox{0.45}{ \begin{tikzpicture}[thick]

\begin{scope}[shift={(0,10)}]
\draw[line width=5] (0,0) .. controls (1,0)  and (5,6).. (0,6); 
\draw[line width=5] (0,0) .. controls (-1,0)  and (-5,6).. (0,6); 
\filldraw (0,0) circle (7pt);
\filldraw (1.5,1.5) circle (7pt);
\filldraw (2.2,2.9) circle (7pt);
\draw (1.55,0.7) node {\Huge$v$};
\draw (2.85,3.35) node {\Huge$u$};
\draw (0,0) .. controls (-6,-3)  and (-5,6).. (1.5,1.5); 
\draw (1.5,1.5) .. controls (3,0.2)  and (4,2.5).. (2.2,2.9); 
\draw (2.2,2.9) .. controls (0,3.5)  and (0,6).. (2,6); 
\draw (2,6) .. controls (8,6)  and (2,-3).. (0,0); 
\draw (0,-1.5) node  {\Huge$S$};
\filldraw[black] (-2.15,2.75) circle (7pt); 
\filldraw[black] (1.25,5.85) circle (7pt); 
\draw (-3,5) node  {\Huge$B$};
\draw (3.3,0) node  {\Huge$R$};
\end{scope}

\begin{scope}[shift={(10,10)}]
\draw (0,0) .. controls (-6,-3)  and (-5,6).. (1.5,1.5); 
\draw (2.2,2.9) .. controls (0,3.5)  and (0,6).. (2,6); 
\filldraw[white] (-2.15,2.75) circle (12pt); 
\filldraw[white] (1.25,5.85) circle (12pt); 
\draw (1.28,1.65) .. controls (3,0.2)  and (4,2.3).. (1.95,2.98); 
\filldraw[white] (1.5,1.5) circle (9pt);
\filldraw[white] (2.2,2.9) circle (9pt);
\draw[line width=5] (0,0) .. controls (1,0)  and (5,6).. (0,6); 
\draw[line width=5] (0,0) .. controls (-1,0)  and (-5,6).. (0,6); 
\filldraw[white] (-0.05,0) circle (12pt);
\draw (2,6) .. controls (8,6)  and (2.3,-2.8).. (-0.47,0.25);
\draw (0,-1.5) node  {\Huge$D_1$};
\end{scope}

\begin{scope}[shift={(20,10)}]
\draw (0,0) .. controls (-6,-3)  and (-5,6).. (1.5,1.5); 
\draw (2.2,2.9) .. controls (0,3.5)  and (0,6).. (2,6); 
\filldraw[white] (-2.15,2.75) circle (12pt); 
\filldraw[white] (1.25,5.85) circle (12pt); 
\draw[line width=5] (0,0) .. controls (1,0)  and (5,6).. (0,6); 
\draw[line width=5] (0,0) .. controls (-1,0)  and (-5,6).. (0,6); 

\filldraw[white] (1.5,1.5) circle (7pt);
\filldraw[white] (2.2,2.9) circle (7pt);
\filldraw[white] (-0.05,0) circle (12pt);
\draw (1.28,1.65) .. controls (3,0.2)  and (4,2.3).. (1.95,2.98); 
\draw (2,6) .. controls (8,6)  and (2.3,-2.8).. (-0.47,0.25); 
\draw (0,-1.5) node  {\Huge$D_2$};
\end{scope}

%

\begin{scope}[shift={(0,-2)}]
\draw (0,0) .. controls (-6,-3)  and (-5,6).. (1.5,1.5); 
\draw (2.2,2.9) .. controls (0,3.5)  and (0,6).. (2,6); 
\filldraw[black] (-2.15,2.75) circle (7pt); 
\filldraw[black] (1.25,5.85) circle (7pt); 
\draw[line width =5] (0,0) .. controls (1,0)  and (5,6).. (0,6); 
\draw[line width =5] (0,0) .. controls (-1,0)  and (-5,6).. (0,6); 
\filldraw (0,0) circle (7pt);
\draw (2,6) .. controls (8,6)  and (2,-3).. (0,0); 
\filldraw[white] (1.5,1.5) circle (7pt);
\filldraw[white] (2.2,2.9) circle (7pt);
\filldraw[white] (1.5,1.5) rectangle  (2.2,2.9) circle;
\draw (1.28,1.65) .. controls (1.58,1.35)  and (2.4,2.98).. (1.95,2.98); 
\draw[line width =5pt] (1.3,1.25) .. controls (2.5,0.8)  and (3.2,2.58).. (2.3,3.2); 
\draw (0,-1.5) node  {\Huge$S'$};
\draw (-3,5) node  {\Huge$B'$};
\draw (3.3,0) node  {\Huge$R'$};

\end{scope}


\begin{scope}[shift={(15,-2)}]
\draw (0,0) .. controls (-6,-3)  and (-5,6).. (1.5,1.5); 
\draw (2.2,2.9) .. controls (0,3.5)  and (0,6).. (2,6); 
\filldraw[white] (-2.15,2.75) circle (12pt); 
\filldraw[white] (1.25,5.85) circle (12pt); 
\draw[line width =5] (0,0) .. controls (1,0)  and (5,6).. (0,6); 
\draw[line width =5] (0,0) .. controls (-1,0)  and (-5,6).. (0,6); 
\filldraw (0,0) circle (3pt);
\filldraw[white] (-0.05,0) circle (12pt);
\draw (2,6) .. controls (8,6)  and (2.3,-2.8).. (-0.47,0.25); 
\filldraw[white] (1.5,1.5) circle (7pt);
\filldraw[white] (2.2,2.9) circle (7pt);
\filldraw[white] (1.5,1.5) rectangle  (2.2,2.9) circle;
\draw (1.28,1.65) .. controls (1.58,1.35)  and (2.4,2.98).. (1.95,2.98); 
\draw[line width =5pt] (1.3,1.25) .. controls (2.5,0.8)  and (3.2,2.58).. (2.3,3.2); 
\draw (0,-1.5) node  {\Huge$D'$};
\end{scope}

\begin{scope}[shift={(15,0)}]
\draw (0,5.5) .. controls (-0.5,7) and (-7,5.2)..  (-7,6.5);
\draw (0,5.5) .. controls (0.5,7) and (7,5.2)..  (7,6.5);
\end{scope}

\end{tikzpicture}}
\caption{The shadow $S'$ is obtained from $S$ by splitting $u$ and $v$ as in Figure~\ref{fig:valid1}. After performing this operation, the straight-ahead cycles $B$ and $R$ of $S$ naturally induce straight-ahead cycles $B'$ and $R'$ of $S'$. Any unknot diagram of $S'$ can be extended to two unknot diagrams of $S$. For instance, the illustrated diagram $D'$ of $S'$ can be extended to the two unknot diagrams $D_1,D_2$ shown. In the digon with endcrossings $u,v$ in $D_1$, the strand corresponding to $B$ is an overstrand, and in $D_2$ it is an understrand. In either case, a Reidemeister move of Type II can be used to transform the diagram into $D'$. Since $D'$ is unknot, then $D_1$ and $D_2$ are also unknot.}
\label{fig:2from1}
\end{figure}

The key claim is that every unknot diagram of $S'$ can be extended to two unknot diagrams of $S$. See Figure~\ref{fig:2from1} for an illustration. For let $D'$ be an unknot diagram of $S'$. We construct from $D'$ two diagrams $D_1, D_2$ as follows. For each vertex in $S$ that is in $S'$, we maintain the prescription it has in $D'$. We obtain $D_1$ (respectively, $D_2$) by giving prescriptions to $u$ and $v$ such that the corresponding crossings are overpasses for the blue (respectively, red) strand.

Both $D_1$ and $D_2$ are unknot diagrams of $S$. This follows simply since, by applying a Reidemeister move of Type II to the strands corresponding to $e$ and $f$, we obtain the unknot diagram $D'$. Thus $D_1$ and $D_2$ are unknot.

This argument shows that $S$ has as least twice as many unknot diagrams as $S'$, and so the lemma follows by induction.
\end{proof}

\begin{proof}[Proof of Lemma~\ref{lem:dum} for even $m$]
A similar reasoning as the one at the beginning of the proof of Lemma~\ref{lem:dum} shows that in this case necessarily $S\neq B\cup R$, and that the respective roots $b$ and $r$ of $B$ and $R$ must be distinct. Since $S\neq B\cup R$, then $S$ has edges that are neither blue nor red; we colour these edges gray.

Let $\oS$ denote the graph that results by removing from $S$ all edges not in $B\cup R$, and supressing all resulting degree $2$ vertices. Thus $\oS$ is the shadow of a link with two components: $\oS$ has exactly two straight-ahead cycles, namely the cycles $\oB$ and $\oR$ naturally inherited from $B$ and $R$.

Note that the roots $b$ and $r$ are not anymore vertices in $\oS$, but we can still regard them as middle points $\obb$ and $\orr$ of edges of $\oB$ and $\oR$, respectively. We say that an unlink diagram of $\oS$ is {\em ($\obb,\orr$)}-{\em avoiding} if it can be taken to the trivial unlink diagram by a series of Reidemeister moves of Type II, none of which involves moving or passing above/below neither $\obb$ nor $\orr$. 

It is easy to see that arguments totally analogous to those in the proof above for the case $m$ odd (but in the present case using Proposition~\ref{prop:dig1} instead of Proposition~\ref{prop:dig2}) yields that $\oS$ has at least $2^{m/2}$ distinct $(\obb,\orr)$-avoiding unlink diagrams.

The key point that we will now show is that each of these $2^{m/2}$ $(\obb,\orr)$-avoiding unlink diagram of $\oS$ can be extended to an unknot diagram of $S$. Given an $(\obb,\orr)$-avoiding unlink diagram $\oD$ of $\oS$, we extend it to a diagram $D$ of $S$ as follows:

\begin{description}
\item{(i)} For each crossing of $D$ that is in $\oD$, we let it maintain in $D$ the prescription it has in $\oD$.

\item{(ii)} If a crossing in $D$ involves a blue (respectively, red) strand and a gray strand, then we let it be an overpass for the blue (respectively, red) strand. Loosely speaking, the blue and red strands are ``above'' the gray strands.

\item{(iii)} For the vertices that are incident with four gray edges (if any) we proceed as follows. Let $S'$ be the shadow that results by removing the edges of $B\cup R$, and supressing all resulting degree $2$ vertices. We give the prescription on the vertices of $S'$ by asking that the resulting assignment on $S'$ is an unknot diagram (take for instance any descending diagram of $S'$).

\item{(iv)} We give arbitrary prescriptions to the crossings corresponding to $b$ and $r$.
\end{description}

These rules give prescriptions to all the vertices in $S$. We claim that the resulting diagram $D$ of $S$ is unknot.

To prove that $D$ is unknot, we start by arguing that the blue-red crossings can be eliminated by using macro moves of Type B. For consider a sequence of $(\obb,\orr)$-avoiding Reidemeister Type II moves that unlink $\oD$. Since the blue and red strands are ``above'' the gray strands (by condition (ii)),  each of this moves can be translated to a macro move of Type B on $D$; this is the step in which it is essential that these moves on $\oD$ are $(\obb,\orr)$-avoiding.

Thus we can eliminate from $D$ all blue-red crossings by a sequence of macro moves of Type B. After this, the remaining blue and red strands are still overstrands (again by condition (ii)), and so they can be collapsed to $\obb$ and $\orr$, respectively, by using macro moves of Type A. Since by (iii) the resulting diagram is unknot, it follows that $D$ is an unknot diagram, as claimed.

Thus each $(\obb,\orr)$-avoiding unlink diagram of $\oS$ can be extended to an unknot diagram of $S$. Since these extensions are all distinct to each other, and the number of $(\obb,\orr)$-avoiding unlink diagrams is at least $2^{m/2}$, it follows that $|\uu(S)|\ge 2^{m/2}$.
\end{proof}

\section{Proof of Theorem~\ref{thm:maintheorem}}\label{sec:proofmain}

Lemmas~\ref{lem:ifone} and~\ref{lem:dum} take care of Theorem~\ref{thm:maintheorem} when $S$ is a shadow that either (i) has a cycle decomposition of size at least $\sqrt[3]{n}$; or (ii) has two straight-ahead cycles with at least $2\sqrt[3]{n}$ common vertices.

Unfortunately, there exist shadows that satisfy neither (i) nor (ii). However, we can show that if $S$ does not satisfy (i), then there is a substructure of $S$ that does satisfy (ii). To make this precise, let us introduce the concept of a subshadow. A shadow $T$ is a {\em subshadow} of $S$ if there is a sequence $S=S_1,S_2,\ldots,$ $S_r=T$ of shadows, such that for $i=1,2,\ldots, r-1$, there is a straight-ahead cycle $C_i$ of $S_i$ such that $S_{i+1}=S_i\cy C_i$. Note that a shadow is a subshadow of itself.

An easy but crucial fact that follows from Proposition~\ref{prop:w1} is that if $T$ is a subshadow of $S$, then $|\uu(S)|\ge |\uu(T)|$. Thus in order to apply Lemma~\ref{lem:dum} to obtain Theorem~\ref{thm:maintheorem}, we do not need the existence of two straight-ahead cycles with these properties in $S$ itself; it suffices to guarantee the existence of such straight-ahead cycles in some subshadow $T$ of $S$. This last ingredient is provided by the following statement.

\begin{lemma}\label{lem:exdu} Let $S$ be a shadow with $n$ vertices. Suppose that every cycle decomposition of $S$ has size at most $\sqrt[3]{n}$. Then there is a subshadow $T$ of $S$ that has straight-ahead cycles $B$ and $R$, such that $B$ and $R$ have at least $2\sqrt[3]{n}$ vertices in common.
\end{lemma}

Before moving on to the proof of this lemma, for completeness let us formally show that with this last ingredient we can finally prove Theorem~\ref{thm:maintheorem}.

\begin{proof}[Proof of Theorem~\ref{thm:maintheorem}]
If $S$ has a cycle decomposition of size at least $\sqrt[3]{n}$, then we are done by Lemma~\ref{lem:ifone}. If every cycle decomposition of $S$ has size at most $\sqrt[3]{n}$, then by Lemma~\ref{lem:exdu} it follows that $S$ has a subshadow $T$ with straight-ahead cycles $B$ and $R$ that have at least $2\sqrt[3]{n}$ common vertices. Applying Lemma~\ref{lem:dum} to $T$, we obtain $|\uu(T)|\ge 2^{\sqrt[3]{n}}$. Since $|\uu(S)|\ge |\uu(T)|$, the theorem follows.
\end{proof}

The proof of Lemma~\ref{lem:exdu} relies on the following two results on cycle decompositions.

\begin{proposition}\label{pro:carito}
Let $S$ be a nontrivial shadow, and let $C$ be a straight-ahead cycle of $S$. Then there is at most one straight-ahead cycle of $S\cy C$ that is not a subdivision of a straight-ahead cycle of $S$.
\end{proposition}

\begin{proof}
Let $S':=S\cy C$, and let $\ff'$ be the set of all straight-ahead cycles of $S'$. Each $F'\in\ff'$ is obtained from a cycle $F$ in $S$, by suppressing from $F$ the vertices that are in $C$ (if any). Let $\ff$ denote the family of all such cycles $F$. It is readily seen that each $F\in\ff$ is a straight-ahead cycle in $S$, unless $F$ contains the root of $C$. Since at most one element in $\ff$ can contain the root of $C$, the proposition follows.
\end{proof}

\begin{proposition}\label{pro:key}
Let $S$ be a nontrivial shadow. Let $\cc=C_1,C_2,\ldots,C_p$ be a cycle decomposition of $S$, and let $D_1,D_2,\ldots,D_p$ be the primary sequence of $\cc$. Then at least two elements in $\{D_1,D_2,\ldots,D_p\}$ are straight-ahead cycles of $S$.
\end{proposition}

\begin{proof}
We proceed by induction on $p$. The base case $p=2$ is straightforward. Thus we take a cycle decomposition $C_1,C_2,\ldots,C_p$ of a shadow $S$, with $p\ge 3$, and assume that the statement holds for all cycle decompositions (of any shadow) of size less than $p$.

Evidently, $C_1$ is a straight-ahead cycle of $S$. Let $S_2:=S\cy C_1$. Consider the cycle decomposition $\cc_2=C_2,C_3,\ldots,C_p$ of $S_2$, and let $E_2,E_3,\ldots,E_p$ be the associated primary sequence. By the induction hypothesis, there exist distinct $j,k\in\{2,\ldots,p\}$ such that $E_j,E_k$ are straight-ahead cycles of $S_2$.

By Proposition~\ref{pro:carito}, at most one of $E_j$ and $E_k$ is not a subdivision of a straight-ahead cycle of $S$. Without loss of generality, $E_j$ is a subdivision of a straight-ahead cycle of $S$. Now $D_j$ is a subdivision of $E_j$, and so it follows that $D_j$ is a straight-ahead cycle in $S$. This completes the proof, since $D_1$ and $D_j$ are straight-ahead cycles in $S$.
\end{proof}

\begin{proof}[Proof of Lemma~\ref{lem:exdu}]

Let $S$ be a shadow with $n$ vertices, and let $C_1,C_2,\ldots,C_p$ be a cycle decomposition of $S$, with $p\le\sqrt[3]{n}$. Let $D_1,D_2,\ldots,D_p$ be the primary sequence of $C_1,C_2,\ldots,C_p$. 

We start by noting that there exist $r,s\in\{1,2,\ldots,p\}$ such that $D_r$ and $D_s$ have at least $2\sqrt[3]{n}$ common vertices. To see this, note that each vertex in $S$ is in the intersection of exactly two cycles in $D_1,D_2,\ldots,D_p$. If each pair of cycles in this sequence had fewer than $2\sqrt[3]{n}$ common vertices, this would imply that $S$ has fewer than $\binom{p}{2}\cdot 2\sqrt[3]{n} \le \binom{\sqrt[3]{n}}{2}\cdot 2\sqrt[3]{n} = n - n^{2/3}$ vertices.

Now let $T$ be a subshadow of $S$ such that (i) $T$ has a cycle decomposition $E_1,E_2,\ldots,E_q$, with primary sequence $F_1,F_2,\ldots,F_q$ such that both $D_r$ and $D_s$ are subdivisions of cycles in $\{F_1,F_2,\ldots,F_q\}$; and (ii) no subshadow of $S$ with fewer vertices than $T$ satisfies (i). Note that $T$ is well defined, as $S$ itself satisfies (i), as witnessed by its cycle decomposition $C_1,C_2,\ldots,C_p$. We claim that $T$ satisfies the conditions in the lemma.

To show this, let $j,k$ be the integers in $\{1,2,\ldots,q\}$ such that $D_r$ is a subdivision of $F_j$, and $D_s$ is a subdivision of $F_k$. 

We proceed by contradiction. Suppose that one of $F_j$ and $F_k$ is not a straight-ahead cycle of $T$. Then, by Proposition~\ref{pro:key}, there is an $\ell\in\{1,2,\ldots,q\}$, $\ell\neq j,k$, such that $F_\ell$ is a straight-ahead cycle of $T$. Let $T':=T\cy F_\ell$. Then $E_1,E_2,\ldots,E_q$ naturally induces a cycle decomposition $E_1',E_2',\ldots,E_{\ell-1}', E_{\ell+1}',\ldots,E_q'$ of $T'$. Indeed,  for each $i\in\{1,2,\ldots,\ell-1,\ell+1,\ldots,q\}$ it suffices to let $E_i'$ be the cycle that results by suppressing from $E_i$ all the vertices that are in $E_\ell$, if any.

Let $F_1',F_2',\ldots,F_{\ell-1}', F_{\ell+1}', \ldots, F_q'$ be the primary sequence of $E_1',E_2',\ldots,E_{\ell-1}', E_{\ell+1}',\ldots,E_q'$. Then, for each $i\in\{1,2,\ldots,\ell-1,\ell+1,\ldots,q\}$ we have that $F_i$ is a subdivision of $F_i'$. In particular, $F_j$ is a subdivision of $F_j'$, and $F_k$ is a subdivision of $F_k'$. It follows that back in $S$, $D_r$ is a subdivision of $F_j'$, and $D_s$ is a subdivision of $F_k'$. Since this violates the minimality of $T$, we conclude that $F_j$ and $F_k$ are straight-ahead cycles of $T$. 
The lemma now follows by setting $B=F_j$ and $R=F_k$, since these two cycles have the same number of common vertices as $D_r$ and $D_s$, which is at least $2\sqrt[3]{n}$. 
\end{proof}


\section{Proofs of Theorem~\ref{thm:trefoil} and Observation~\ref{obs:figure8}}\label{sec:trefoil}

Let us start by recalling the concept of a cut-vertex in a graph. Following Tutte~\cite{tut}, a $1$-{\em separation} of a connected graph $G$ is an ordered pair $(H,K)$ of subgraphs of $G$, each with at least one edge, such that $H\cup K=G$ and $H\cap K$ is a graph that consists of a single vertex. The vertex of $H\cap K$ is the {\em cut-vertex} of the $1$-separation.

We remark that sometimes a cut-vertex of a connected graph is thought of as a vertex whose deletion produces a disconnected graph. We note that Tutte's notion of a cut-vertex is more general. Consider for instance vertex $v$ in the shadow at the top of Figure~\ref{fig:chorizo}. If we delete this vertex, together with its incident edges, the resulting graph is still connected, and yet $v$ is a cut-vertex according to Tutte's definition. Under Tutte's definition, if a graph $G$ can be obtained from two disjoint graphs, taking one vertex of each of these graphs and identifying these two vertices into a vertex $v$, then $v$ is a cut-vertex of $G$.

\subsection{Proof of Theorem~\ref{thm:trefoil}}\label{subsec:prooftrefoil}

We prove the two statements in Theorem~\ref{thm:trefoil} separately.

\bigskip
\noindent{\sl (I) Every diagram of $S$ is unknot if and only if every vertex of $S$ is a cut-vertex.}
\bigskip

For brevity, let us say that a shadow is {\em simple} if each of its vertices is a cut-vertex.

For the easier direction of (I), we note that if $D$ is a diagram of a simple shadow, then every crossing in $D$ is nugatory. We recall that a {\em nugatory} crossing in a diagram corresponds precisely to a cut-vertex in its shadow. This crossing can be trivialized by a switching operation of the diagram, thus obtaining an equivalent diagram with fewer crossings. If every crossing in a diagram is nugatory, then it is equivalent to the trivial diagram. Thus, if $S$ is simple, then all the diagrams of $S$ are unknot.

We prove the other direction of (I) by induction on the number of vertices of $S$. In the base case $S$ consists of exactly one vertex, and the statement is trivial. For the inductive step, we assume that the statement holds for all simple shadows with fewer than $n$ vertices, for some $n\ge 2$, and let $S$ be a simple shadow with $n$ vertices. 

First we note that if $S$ has no cut-vertex, then not all the diagrams of $S$ are unknot. To see this, let $D$ be an alternating diagram of $S$, that is, a diagram in which as one traverses the diagram, one encounters an overcrossing, then an undercrossing, then an overcrossing, and so on. Since $S$ has no cut vertices, then $D$ has no nugatory crossings, that is, $D$ is a {\em reduced} diagram. The validity~\cites{kauf1,muras1,this1} of the First Tait Conjecture (if $D$ is a reduced alternating diagram, then no diagram equivalent to $D$ has fewer crossings than $D$) then implies that the knot represented by $D$ cannot be the unknot.

We can then assume that $S$ has a cut-vertex $v$. Then there are subgraphs $H,K$ of $S$, both with at least one edge, such that $H\cup K=S$ and $H\cap K$ consists solely of $v$. Note that $v$ has degree $2$ in both $H$ and $K$. Let $H'$ (respectively, $K'$) be the graph obtained by suppressing $v$ in $H$ (respectively, $K$). Then $H'$ and $K'$ are shadows, each with fewer than $n$ vertices. 

We claim that both $H'$ and $K'$ are simple. For suppose that $H'$ is not simple, and let $D_H$ be a knotted diagram of $H'$. Let $D_K$ be an unknot diagram of $K'$. We can naturally combine $D_H$ and $D_K$ into a diagram $D$ of $S$, by giving any prescription to $v$. Since $D_K$ is unknot, it follows that $D$ and $D_H$ are equivalent, and so $D$ is a knotted diagram of $S$. Since this contradicts the assumption that $S$ is simple, we conclude that $H'$ must be simple. The same argument shows that $K'$ is simple.

The inductive hypothesis then implies that every vertex of $H'$ and $K'$ is a cut-vertex, and from this it immediately follows that every vertex of $S$ is a cut-vertex.  \hfill \qed

\bigskip
\noindent{(II) {\sl If not every vertex of $S$ is a cut-vertex, then $S$ has a trefoil diagram.}}
\bigskip

It is easy to see that if every straight-ahead cycle of a shadow has only one vertex (that is, it consists of a vertex and a loop-edge), then every vertex is a cut-vertex of the shadow. From (I) it then follows that if not all the vertices of a shadow are cut vertices, then the shadow has a straight-ahead cycle with more than one vertex. 


The proof is by induction on the number of vertices of $S$. A quick analysis shows that in a shadow with $1$ or $2$ vertices, all the vertices are cut vertices. It is also readily seen that up to isomorphism there is only one shadow with $3$ vertices, such that not all of its vertices are cut vertices. This is the shadow of the familiar diagram of the trefoil knot. This shadow has evidently a trefoil diagram, and so this settles the base case $n=3$. 

For the inductive step, we take a shadow $S$ on $n\ge 4$ vertices, where not all the vertices of $S$ are cut vertices, and assume that (II) holds for shadows with fewer than $n$ vertices.

Let $S$ be a shadow in which not every vertex of $S$ is a cut-vertex. It follows from the observation above that $S$ has a straight-ahead cycle with more than one vertex. Let $C$ be such a straight-ahead cycle, and let $r$ be its root. Let $W$ be the other straight-ahead closed walk that starts and ends at $r$. Note that $S=C\cup W$. 

The walk $W$ can then be naturally written as the concatenation of three walks $W_1,W_2$, and $W_3$, as follows. Since $C$ has at least one vertex other than $r$, and $W$ is a closed walk and $S=C\cup W$, then $W$ and $C$ must have at least two vertices in common, other than $r$. 

Let $u$ be the first vertex of $C$ (other than $r$) that we find as we traverse $W$. After reaching $u$, we continue traversing $W$; let $v$ be the first vertex in $C$ that we find in this continuation of the traversal of $W$. Then let $W_1$ be the part of $W$ from $r$ to $u$, let $W_2$ be the part of $W$ from $u$ to $v$, and let $W_3$ be the part from $v$ to $r$. Note that $W_1$ is contained in a component $\Delta_1$ of $\real^2\setminus C$, and $W_2$ is contained in the other component $\Delta_2$. The end edges of $W_3$ are contained in $\Delta_1$, but some parts of $W_3$ may be in $\Delta_1$.

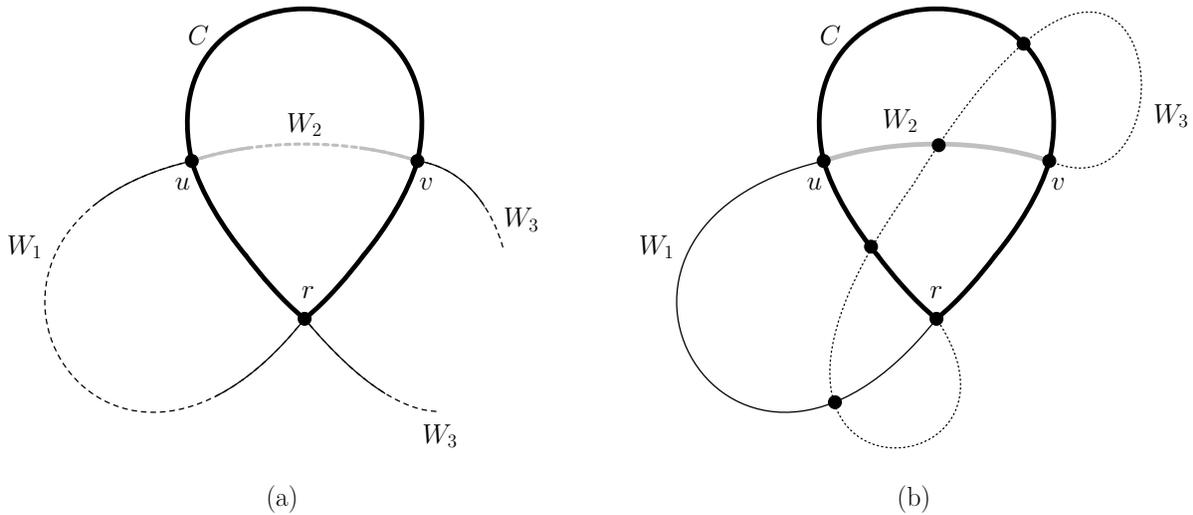
\begin{figure}[ht!]
\vglue -2.5 cm\scalebox{0.6}{\begin{tikzpicture}[thick, line cap=round, line join=round]
\hglue -2 cm
\begin{scope}[shift={(-7,0)}]
\draw[line width=3] (0,-2) .. controls (-0.5,-1.6) and (-1,-1) ..(-1.3,-0.6);
\draw[line width=3] (-1.3,-0.6) .. controls (-1.8,0) and (-2.3,0.8) ..(-2.5,1.5);
\draw[line width=3] (0,-2) .. controls (0.5,-1.6) and (1,-1) ..(1.3,-0.6);
\draw[line width=3] (1.3,-0.6) .. controls (1.8,0) and (2.3,0.8) ..(2.5,1.5);
\draw[line width=3] (2.5,1.5) .. controls (3.5,6) and (-3.5,6) .. (-2.5,1.5);

\draw(-2.5,1.5) .. controls (-9,0) and (-4.5,-7.7) .. (0,-2);
\filldraw[white] (-6.67,-4) rectangle (-4,0.5);
\filldraw[white] (-6.67,-4.5) rectangle (-2,-3);
\draw[dashed](-2.5,1.5) .. controls (-9,0) and (-4.5,-7.7) .. (0,-2);
\draw[line width=2,lightgray] (-2.5,1.5) .. controls (-1,2)  and (1,2) .. (2.5,1.5);
\filldraw[white] (-1.2,3) rectangle (1.2,1.5);
\draw[line width=2,lightgray,dashed] (-2.5,1.5) .. controls (-1,2)  and (1,2) .. (2.5,1.5);
\draw[] (2.5,1.5) .. controls  (6,1) and (5,-8)  ..   (0,-2);
\filldraw[white] (4,2) rectangle (5,-6);
\filldraw[white] (1.8,-2) rectangle (5,-6);
\draw[dashed] (2.5,1.5) .. controls  (6,1) and (5,-8)  ..   (0,-2);
\filldraw[white] (3,-0.5) rectangle (5,-4.2);
%
\filldraw[] (2.5,1.5) circle (4pt);
\filldraw[] (-2.5,1.5) circle (4pt);
\filldraw[] (0,-2) circle (4pt);
\draw  (0.08,-1.4) node {\LARGE$r$};
\draw  (0,2.3) node {\LARGE$W_2$};
\draw  (4.8,0.2) node {\LARGE$W_3$};
\draw  (3.,-4.6) node {\LARGE$W_3$};
\draw  (-6.2,-0.4) node {\LARGE$W_1$};
\draw  (2.7,1.0) node {\LARGE$v$};
\draw  (-2.7,1.0) node {\LARGE$u$};
\draw  (-2.35,4.3) node {\LARGE$C$};
\draw  (-0.5,-6) node {\LARGE(a)};
\end{scope}
\begin{scope}[shift={(7,0)}]
\draw[line width=3] (0,-2) .. controls (-0.5,-1.6) and (-1,-1) ..(-1.3,-0.6);
\draw[line width=3] (-1.3,-0.6) .. controls (-1.8,0) and (-2.3,0.8) ..(-2.5,1.5);
\draw[line width=3] (0,-2) .. controls (0.5,-1.6) and (1,-1) ..(1.3,-0.6);
\draw[line width=3] (1.3,-0.6) .. controls (1.8,0) and (2.3,0.8) ..(2.5,1.5);
\draw[line width=3] (2.5,1.5) .. controls (3.5,6) and (-3.5,6) .. (-2.5,1.5);
\draw[line width=3,lightgray] (-2.5,1.5) .. controls (-1,2)  and (1,2) .. (2.5,1.5);
\draw(-2.5,1.5) .. controls (-9,0) and (-4.5,-7.7) .. (0,-2);
\draw[dotted] (2.5,1.5) .. controls  (6,0) and (5,10)  ..  (-0.50,1) .. controls (-6,-6.5) and (2.7,-6) .. (0,-2);
\filldraw[] (0,-2) circle (4pt);
\filldraw[] (0.05,1.85) circle (4pt);
\filldraw[] (2.5,1.5) circle (4pt);
\filldraw[] (-2.5,1.5) circle (4pt);
\filldraw[] (1.93,4.1) circle (4pt);
 \filldraw[] (-1.45,-0.4) circle (4pt);
 \filldraw[] (-2.25,-3.85) circle (4pt);
\draw  (0,-1.4) node {\LARGE$r$};
\draw  (-0.8,2.4) node {\LARGE$W_2$};
\draw  (5.2,2.5) node {\LARGE$W_3$};
\draw  (2.7,1.0) node {\LARGE$v$};
\draw  (-2.7,1.0) node {\LARGE$u$};
\draw  (-6.2,-0.4) node {\LARGE$W_1$};
\draw  (-2.35,4.3) node {\LARGE$C$};
\draw  (-0.5,-6) node {\LARGE(b)};

\end{scope}
\end{tikzpicture}}
\caption{If $S$ is a shadow with a straight-ahead cycle $C$ that has more than one vertex, then $S$ must have roughly the structure shown in (a). That is, the complement of $C$ in $S$ can be naturally decomposed into three parts $W_1,W_2$ and $W_3$, where $W_1$ and $W_2$ are contained in distinct components of $\real^2\setminus C$, and $W_3$ need not be contained in any of these components, but its end edges are contained in the same component as $W_1$. In part (b) of this figure we depict a whole shadow that satisfies these properties.}
\label{fig:concat}
\end{figure}

The situation is illustrated in Figure~\ref{fig:concat}(a). The dashed parts of $W_1$ and $W_2$ hint to the fact that $W_1$ and $W_2$ may not be paths. The dashed parts of $W_3$ suggest that the beginning and the end of $W_3$ are in $\Delta_1$, but other parts of $W_3$ may be in $\Delta_2$.

We now show that if one of $W_1,W_2,W_3$ is not a path, then we are done by the induction hypothesis. For suppose that $W_1$ is not a path. Since $W_1$ is a straight ahead-walk, it follows that there must exist a straight-ahead cycle $F$ contained in $W_1$. We note that none of  $r,u$, and $v$ can be in $F$. Now let $T:=S\cy F$. The shadow $T$ has then a straight-ahead cycle (namely the one naturally induced by $C$) with more than one vertex, and so by the induction hypothesis $T$ has a trefoil diagram. But this trefoil diagram can be extended to an equivalent diagram of $S$, by having the strand corresponding to $F$ to be an overstrand. Therefore we conclude that $S$ has a trefoil diagram.

In exactly the same way it is shown that if one of $W_2$ and $W_3$ is not a path, then we are done by the induction hypothesis. We can then assume that $W_1,W_2$, and $W_3$ are paths. A possible scenario is illustrated in Figure~\ref{fig:concat}(b). We finish the proof by showing that $S$ has a trefoil diagram.

Consider a diagram $D$ of $S$ in which the strand corresponding to $W_3$ is an overstrand. An example is illustrated in Figure~\ref{fig:tref2}(b), which is a diagram of the shadow in Figure~\ref{fig:concat}.

With a macro move of Type B we can replace this strand with a strand that is contained in $\Delta_2$, and has no internal crossings. Let $D'$ be the diagram thus obtained. (See Figure~\ref{fig:tref2}(b)). The shadow of $D'$ is then isomorphic to the usual shadow of the trefoil knot, and so it follows that by giving suitable prescriptions at $r,u$, and $v$ (as in Figure~\ref{fig:tref2}(b)), we obtain that $D'$ is a diagram of the trefoil knot. By giving these prescriptions at $r,u$, and $v$ also in $D$, it follows that $D$ and $D'$ are equivalent, and so $D$ is a trefoil diagram of $S$. \hfill{\qed}

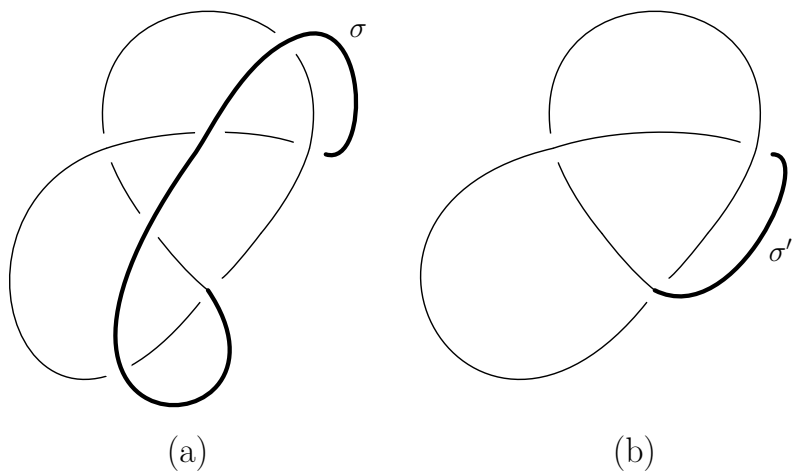
\begin{figure}[ht!]
\centering
\hglue - 2cm\scalebox{0.54}{\begin{tikzpicture}[line width= 1, line cap=round, line join=round]

\begin{scope}[shift={(0,0)}]
\draw(0,-2) .. controls (-0.5,-1.6) and (-1,-1) ..(-1.3,-0.6);
\draw(-1.3,-0.6) .. controls (-1.8,0) and (-2.3,0.8) ..(-2.5,1.5);
\draw(0.35,-1.7) .. controls (0.5,-1.6) and (1,-1) ..(1.3,-0.6);
\draw (1.3,-0.6) .. controls (1.8,0) and (2.3,0.8) ..(2.5,1.5);
\draw (2.5,1.5) .. controls (3.5,6) and (-3.5,6) .. (-2.5,1.5);
\filldraw[white] (-2.5,1.5) circle (10pt);
\draw (-2.5,1.5) .. controls (-1,2)  and (1,2) .. (2.1,1.65);
\draw(-2.5,1.5) .. controls (-7,0) and (-4.5,-7.7) .. (-0.2,-2.3);
\filldraw[white] (1.93,4.1) circle (10pt);
 \filldraw[white] (-1.45,-0.4) circle (10pt);
 \filldraw[white] (-2.25,-3.85) circle (10pt);
\filldraw[white] (0.05,1.85) circle (10pt);
\draw[line width= 3] (2.9,1.35) .. controls  (4,1) and (4,5.1) ..( 2.1,4.2) .. controls (0.9,3.7) and (0.1,2)  .. (-0.30,1.4) .. controls (-6,-6.5) and (2.7,-6) .. (0,-2);
\draw  (-0.5,-6) node {\Huge(a)};
\draw  (3.7,4.4) node {\huge$\sigma$};

\end{scope}
\begin{scope}[shift={(11,0)}]
\draw(0,-2) .. controls (-0.5,-1.6) and (-1,-1) ..(-1.3,-0.6);
\draw (-1.3,-0.6) .. controls (-1.8,0) and (-2.3,0.8) ..(-2.5,1.5);
\draw (0.35,-1.7) .. controls (0.5,-1.6) and (1,-1) ..(1.3,-0.6);
\draw (1.3,-0.6) .. controls (1.8,0) and (2.3,0.8) ..(2.5,1.5);
\draw (2.5,1.5) .. controls (3.5,6) and (-3.5,6) .. (-2.5,1.5);
\filldraw[white] (-2.5,1.5) circle (10pt);
\draw(-2.5,1.5) .. controls (-1,2)  and (1,2) .. (2.1,1.65);
\draw(-2.5,1.5) .. controls (-9,0) and (-4.5,-7.7) .. (-0.2,-2.3);
\draw[line width=3] (2.9,1.35) .. controls  (4,1.35) and (2,-3)  .. (0,-2);
\draw  (-0.5,-6) node {\Huge(b)};
\draw  (3.1,-1.0) node {\huge$\sigma'$};

\end{scope}
\end{tikzpicture}}
\caption{On the left hand side we have a diagram of a shadow in which $W_1,W_2$, and $W_3$ are paths (this is actually a diagram of the shadow in Figure~\ref{fig:concat}(b)), and the strand $\sigma$ corresponding to $W_3$ is an overstrand. We can use a macro move of Type B to shift $\sigma$ to a strand $\sigma'$ with no internal crossings, to obtain the equivalent diagram shown in (b). This last diagram is clearly a trefoil diagram, and so the diagram in (a) is a trefoil diagram of the shadow in Figure~\ref{fig:concat}(b).}
\label{fig:tref2}
\end{figure}

%
\subsection[Shadows that do not yield any knot with even crossing number: \\ proof of Observation~\ref{obs:figure8}]{Proof of Observation~\ref{obs:figure8}}\label{sec:f8}
%
For each integer $n \ge 3$, let $C_n$ be the shadow obtained by taking a cycle with $n$ vertices and adding a parallel edge to each edge. In Figure~\ref{fig:circulant} we illustrate $C_7$.

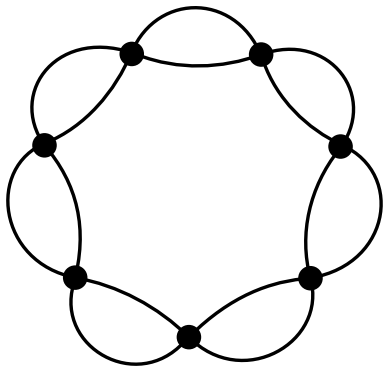
\begin{figure}[ht!]
\scalebox{0.8}{\begin{tikzpicture}[line width=1.6, line cap=round, line join=round]
\begin{scope}
    \draw(4.33, 1.88) .. controls (4.04, 0.84) and (5.38, 0.08) .. (6.16, 0.92);
    \draw (6.16, 0.92) .. controls (6.81, 1.60) and (7.46, 1.91) .. (8.15, 1.98);
    \draw (8.15, 1.98) .. controls (9.59, 2.12) and (9.86, 3.72) .. (8.80, 4.20);
    \draw (8.80, 4.20) .. controls (8.20, 4.48) and (7.71, 4.97) .. (7.49, 5.59);
    \draw (5.30, 5.75) .. controls (5.02, 5.09) and (4.53, 4.51) .. (3.87, 4.24);
    \draw(3.87, 4.24) .. controls (2.88, 3.84) and (3.07, 2.25) .. (4.36, 2.00);
    \draw (4.36, 2.00) .. controls (5.08, 1.87) and (5.73, 1.52) .. (6.25, 1.01);
    \draw (6.25, 1.01) .. controls (7.14, 0.11) and (8.56, 0.92) .. (8.27, 1.99);
    \draw (8.27, 1.99) .. controls (8.07, 2.73) and (8.26, 3.51) .. (8.73, 4.11);
    \draw (8.73, 4.11) .. controls (9.43, 5.00) and (8.59, 6.12) .. (7.45, 5.71);
    \draw (5.41, 5.70) .. controls (4.09, 6.22) and (3.22, 5.09) .. (3.85, 4.2);
    \draw (3.85, 4.2) .. controls (4.39, 3.57) and (4.59, 2.82) .. (4.33, 1.88);
 \draw [](7.49, 5.59) .. controls (7.10, 6.75) and (5.73, 6.77) .. (5.30, 5.75);
\draw [](7.45, 5.71) .. controls (6.79, 5.47) and (6.07, 5.45) .. (5.41, 5.70);
\filldraw [black] (3.85,4.2) circle [radius=5pt];
\filldraw [black] (4.36,2) circle [radius=5pt];
\filldraw [black] (6.25,1.01) circle [radius=5pt];
\filldraw [black] (8.27,1.99) circle [radius=5pt];
\filldraw [black] (8.77,4.18) circle [radius=5pt];
\filldraw [black] (7.45,5.71) circle [radius=5pt];
\filldraw [black] (5.3,5.72) circle [radius=5pt];
  \end{scope}
\end{tikzpicture}}
\caption{The shadow $C_7$. No diagram associated to $C_7$ is equivalent to a diagram of a knot with even crossing number (and in particular to the figure-eight knot). }
\label{fig:circulant}
\end{figure}

To prove Observation~\ref{obs:figure8} it suffices to show that for every odd integer $n$, if $D$ is a diagram of $C_n$, then the knot corresponding to $D$ is either trivial, or it has odd crossing number.

We prove this by induction on $n$. For the base case $n=3$ it is readily seen that every diagram of $C_3$ is either a trefoil diagram or an unknot diagram. Suppose then that $n\ge 5$, and that the statement holds for all odd integers smaller than $n$. 

Let $D$ be a diagram of $C_n$. Since $C_n$ is a $2$-connected graph, it follows that $D$ is a reduced diagram. If $D$ is alternating then, since $D$ is reduced, the crossing number of the knot corresponding to $D$ is $n$~\cites{kauf1,muras1,this1}. Suppose finally that $D$ is not alternating. Then there are two neighbour crossings on which we can perform a Reidemeister move of Type II on $D$, obtaining a diagram of $C_{n-2}$, and the result follows by the induction hypothesis. \hfill\qedsymbol

\section{Concluding remarks and open questions}\label{sec:concludingremarks}

Before starting this project, we considered the following question. Given a shadow $S$ on $n$ vertices, how many different knot types are there among the $2^n$ distinct diagrams of $S$? This question was posed by the third author in {\tt mathoverflow.net}~\cite{myover}. Among the replies received, Andy Putnam mentioned shadows that have only unknot diagrams. After obtaining the easy characterization of which shadows have this property (the first part of Theorem~\ref{thm:trefoil}), we turned our attention to the main topic in this paper.

We find the question in the previous paragraph quite interesting. We leave aside those shadows that only have unknot diagrams, as they admit an easy characterization. The graphs used in the proof of Observation~\ref{obs:figure8} show that for each odd integer $n$, there is a $2$-connected shadow $C_n$ on $n$ vertices, such that among the $2^n$ distinct diagrams of $C_n$ there are only $\lceil{n/2}\rceil$ distinct knot types. Now, even though very large, these shadows have very small depth~\cite{erickson}.

We adapt the notion of depth of a simple curve, to the parallel context of the depth of a shadow. The {\em depth} of a face $f$ of a shadow $S$ is the length of a minimum distance path, in the dual graph of $S$, from the dual vertex of $f$ to the dual vertex of the unbounded face. The {\em depth} of $S$ is then the maximum depth over all faces of $S$.

Is it true that one can guarantee many more knot types for $2$-connected shadows with large depth? We note that the $2$-connectedness requirement is important, as the shadow on top of Figure~\ref{fig:chorizo} can be generalized to obtain a shadow with arbitrarily large depth, and that only has unknot diagrams. Perhaps some of the tools developed in~\cite{erickson} can shed light on this question. 

It seems natural to ask if one could hope to prove a much better bound than the one given by Theorem~\ref{thm:maintheorem}. To explore this question, we let $\umin(n)$ denote the best possible bound. That is, for each positive integer $n$,
\[
\umin(n):=\min_{|V(S)|=n} |\uu(S)|,
\]
where the minimum is taken over all shadows with $n$ vertices.

Under this terminology, Theorem~\ref{thm:maintheorem} gives the superpolynomial bound $\umin(n)\ge 2^{\sqrt[3]{n}}$. Leaving aside relatively marginal possible improvements, such as showing that $\umin(n)\ge 2^{\sqrt{n}}$, the important question seems to be the following.

\begin{question}\label{que:subexpo}
Is $\umin(n)$ an exponential, or a subexponential function?
\end{question}

It follows from the work in~\cites{chapman1,chapman2} that there is a constant $d$, $1<d<2$, such that $\umin(n) < d^n$. To find explicit upper bounds for $d$, we can proceed as follows. Consider any fixed shadow $T$ with $m$ vertices, and let $t:=|\uu(T)|$. Now let $T^k$ denote the connected sum $T^k$ of $k$ copies of $T$. (One can easily extend to shadows the definition of the connected sum of diagrams). Since the connected sum of two diagrams is unknot if and only if each of the diagrams is unknot, it follows that $T^k$ has exactly $t^k$ unknot diagrams. Thus $T^k$ is a shadow with $n:=km$ vertices and $|\uu(T^k)|=t^k=(t^{{k}/{n}})^n=(t^{1/m})^n$. From this it follows that $\umin(n)\le t^{1/m}$.

To apply this approach we have explored a variety of shadows for which we can estimate (upper bound) their number of unknot diagrams. Using SnapPy~\cite{SnapPy} we found that there is a shadow $S$ on $16$ vertices such that $|\uu(S)| < 6416$. From the previous discussion it follows that, for every positive integer $n$ divisible by $16$, $\umin(n)\le ({6416}^{1/16})^n\approx (1.729)^n$.

Theorem~\ref{thm:trefoil} gives a structural characterization of which shadows have a trefoil diagram, and Observation~\ref{obs:figure8} shows that there are arbitrarily large $2$-connected shadows that do not have a figure-eight diagram. It seems natural to inquire what can be said about the figure-eight knot:

\begin{question}
Is there a simple characterization of which shadows have a figure-eight diagram? Is it true that every sufficiently large $3$-connected shadow contains a diagram of the figure-eight knot? 
\end{question}

\section*{Acknowledgements}

This work started while the second author visited the other authors, supported by the Laboratorio Internacional Asociado Solomon Lefschetz (LAISLA), which has now become the Unit{\'e} Mixte Internationale CNRS-CONACYT-UNAM ``Laboratoire Solomon Lefschetz''. The first and third author were supported by Conacyt grant 222667. We thank Ronald Ortner for helpful comments, and Mario Eudave for his guidance at the beginning of this project.


\end{document}